\documentclass[a4,12pt,reqno]{amsart}
\usepackage{amsmath, amssymb, amsthm, a4wide}
\usepackage{graphicx}
\usepackage{dsfont}
\usepackage{url}
\usepackage[latin1]{inputenc}
\usepackage{multirow}
\usepackage[font=footnotesize]{subfig}
\usepackage{array}
\usepackage{mathdots}
\usepackage{color}
\usepackage[colorlinks,citecolor=blue,urlcolor=blue]{hyperref}

\newcommand{\F}{\text{F}}

\newcommand{\e}{\mathrm{e}}
\newcommand{\W}{\mathbf{W}}

\newcommand{\R}{\mathbb{R}}

\newcommand{\C}{\mathbb{C}}
\newcommand{\E}{\mathbb{E}}

\newcommand{\I}{\mathbf{I}}

\newcommand{\dif}{\mathrm{d}}

\newtheoremstyle{definition}%
  {0.5em}
  {\topsep}
  {\upshape}
  {}
  {\bfseries}
  {.}
  { }
  {}

\theoremstyle{plain}
\newtheorem{theorem}{Theorem}

\newtheorem{lemma}{Lemma}
\newtheorem{corollary}{Corollary}
\newtheorem{remark}{Remark}
\theoremstyle{definition}
\newtheorem{assumption}{Assumption}

\DeclareMathAlphabet\mathbfcal{OMS}{cmsy}{b}{n}

\allowdisplaybreaks[2]

\begin{document}


%
%

\begin{center}
{\Large
    {\sc Hypergeometric Functions of Matrix Arguments and Linear Statistics of Multi-Spiked Hermitian Matrix Models}
}
\bigskip

Damien Passemier$^{1}$ \& Matthew R. McKay$^{2}$ \& Yang Chen$^{3}$
\bigskip

{\it
$^{1,2}$Department of Electronic and Computer Engineering\\
Hong Kong University of Science and Technology (HKUST)\\
Clear Water Bay, Kowloon, Hong Kong\\
$^{1}$damien.passemier@gmail.com $^{2}$eemckay@ust.hk\\
\bigskip
$^{3}$Faculty of Science and Technology,
Department of Mathematics\\
University of Macau,
Avenue Padre Tom\'as Pereira, Taipa Macau, China\\
yangbrookchen@yahoo.co.uk\\

}
\end{center}
\bigskip

{\bf Abstract.} This paper derives central limit theorems (CLTs) for general linear spectral statistics (LSS) of three important multi-spiked Hermitian random matrix ensembles. The first is the most common spiked scenario, proposed by Johnstone, which is a central Wishart ensemble with fixed-rank perturbation of the identity matrix, the second is a non-central Wishart ensemble with fixed-rank noncentrality parameter, and the third is a similarly defined non-central $F$ ensemble.  These CLT results generalize our recent work \cite{Passemier2} to account for multiple spikes, which is the most common scenario met in practice.  The generalization is non-trivial, as it now requires dealing with hypergeometric functions of matrix arguments.  To facilitate our analysis, for a broad class of such functions, we first generalize a recent result of Onatski \cite{Onatski3} to present new contour integral representations, which are particularly suitable for computing large-dimensional properties of spiked matrix ensembles.  Armed with such representations, our CLT formulas are derived for each of the three spiked models of interest by employing the Coulomb fluid method from random matrix theory along with saddlepoint techniques.  We find that for each matrix model, and for general LSS, the individual spikes contribute additively to yield a $O(1)$ correction term to the asymptotic mean of the linear statistic, which we specify explicitly, whilst having no effect on the leading order terms of the mean or variance.

\smallskip


\section{Introduction}

Linear spectral statistics (LSS) are of fundamental importance in multivariate analysis. Such statistics are characterized, quite generally, by summations of functions of the individual eigenvalues of random matrices.  Of prime interest  are sample covariance matrices, constructed based on $m$ observations (samples) of a $n$-dimensional random vector (variables), or suitably defined $F$ matrices.  Many classical results are available which specify the asymptotic distribution of certain LSS as the number of observations $m$ become asymptotically large, for fixed $n$ (see., e.g., \cite{Anderson,Fujikoshi}).  However, modern applications often deal with high-dimensional data sets, for which $n$ and $m$ have similar order, and thus classical asymptotics no longer apply.  This has inspired a new wave of research, aimed at characterizing the limiting distributions of LSS in the double-asymptotic regime, with $n$ and $m$ both large, using tools from asymptotic random matrix theory.  The asymptotic behavior is typically found to be markedly different from the classical asymptotic setting, whilst giving substantially improved accuracy for various practical applications.  

Under double-asymptotics, central limit theorems (CLTs) for LSS of various random matrix ensembles have now been derived, providing generic asymptotic formulas for the limiting mean and variance (see, e.g., \cite{Chen,Diaconis,Lytova,Zheng}). Much of this attention has focused on scenarios with identity population covariance (e.g., \cite{Chen,Lytova,AndersonCLT}), although some results for more general matrix models have also appeared \cite{Bai04}.  

In this paper, our main focus is on three related classes of so-called ``spiked'' Hermitian random matrix ensembles: (i) central Wishart with finite-rank perturbation of the identity (proposed in \cite{Johnstone}, which we refer to as ``Johnstone's spiked model''), (ii) non-central Wishart with fixed-rank noncentrality, and (iii) similarly defined non-central $F$ matrices.  In our recent work \cite{Passemier2}, we demonstrated that each of these models shared a common feature, with their joint eigenvalue distributions being expressible in a similar contour-integral form. These were initially derived in \cite{Forrester,Wang,Mo,Onatski} and \cite{Dharmawansa} for Johnstone's spiked model and the non-central Wishart spike model respectively, whilst given as a new result in \cite{Passemier2} for the spiked $F$ model. If one disregards the contour integral, then the joint eigenvalue density in each case yields the same general structure as typical of ``classical'' Hermitian random matrix ensembles (i.e., the Gaussian, Laguerre and Jacobi unitary ensembles \cite{Mehta}), albeit with a modified weight function.  This is particularly important, as it allows one to  employ powerful methods designed for such classical ensembles in the study of spiked models.  This was precisely the approach undertaken in \cite{Passemier2}, where we employed the framework \cite{Chen}, designed for non-spiked models based on Dyson's Coulomb fluid method \cite{Dyson} (see also \cite{Chen+Manning, chemancond, Chen+Ismail,  BasorChen, Simon2006, Vivo2008, Dean2008}), as well as saddlepoint integration techniques, to derive CLTs for each of the three spiked models under consideration as the matrix dimensions grew large. 

Our recent results in \cite{Passemier2} assumed the presence of a single spike only, which for many practical applications may not be reasonable. Examples include \cite{Anderson} in psychology, \cite{Kritchman2,Nadakuditi,Dumont,Bianchi,CouilletBook,Couillet} in signal processing, \cite{Naes, Kritchman} in physics of mixture, \cite{Bouchaud,Torun} in finance, \cite{Dahirel, Quadeer2013} in computational immunology/virology, \cite{Bai09,Bai13} in statistics, in addition to others (e.g., \cite{Johnstone,Passemier1}). 
For scenarios with multiple spikes, there are relatively few existing results concerning  LSS, and the results which are available focus primarily on Johnstone's spiked model.  These include \cite{Wang3}, which derived a CLT for general LSS, expressing the limiting mean and variance in terms of contour integrals,  as well as \cite{Onatski3,Passemier,Wang2,Onatski2}, which considered specific linear statistics (i.e., for specific applications).  For alternative spiked models, such as the non-central Wishart and $F$ scenarios, results concerning LSS are currently absent, beyond the single-spike scenario considered in \cite{Passemier2}.


The primary objective of this paper is to close this gap by deriving CLTs for general LSS under all three spiked models indicated above, allowing for arbitrary numbers of spikes.  This generalization is substantial, since one can no longer rely on the contour-integral-based joint eigenvalue densities in \cite{Wang,Mo,Onatski} and \cite{Dharmawansa}.  Thus, the first major step is to obtain expressions for the joint eigenvalue densities which, for the three spiked models, are classically expressed in terms of hypergeometric functions of matrix arguments \cite{James}. Such functions are fundamental objects arising in multivariate analysis and random matrix theory, and are often difficult to handle.  Quite generally, they are denoted
\begin{align} \label{eq:Def}
{}_p F_q^{(\alpha)} \left( a_1, \ldots, a_p; b_1, \ldots, b_q ; \mathbf{X}, \mathbf{Y}   \right),  
\end{align}
where $\mathbf{X}$ and $\mathbf{Y}$ are $n \times n$ Hermitian matrix arguments. For $\alpha = 2$, $1$, and $1/2$, these functions are often described as solutions to matrix integrals over the orthogonal, unitary, or symplectic groups respectively (see, for example, \cite{Muirhead,Gupta}, and \cite{James}); alternatively, they may be described via infinite series expansions involving Jack polynomials.  In addition, for the special case $\alpha = 1$ (only), they may also be written as $n \times n$ determinants of scalar ${}_p F_q$ functions \cite{Khatri70, Gross89}. Whilst computational algorithms have been developed recently (e.g., \cite{Koev}),  such representations are still difficult to describe when the matrix dimension $n$ is large, which is often the realm of interest for problems in random matrix theory.  

For each of our three spiked models of interest, the joint eigenvalue density involves a specific particularization of (\ref{eq:Def}): For Johnstone's spiked model it is a ${}_0 F_0^{(1)}$, for the non-central Wishart spiked model it is a ${}_0 F_1^{(1)}$, whilst for the non-central $F$ spiked model it is a ${}_1F_1^{(1)}$.   However, it turns out that for all three cases, one of the matrix arguments in the associated hypergeometric function has rank $r \leq n$, with $r$ denoting the number of spikes. For such a reduced-rank scenario, in some exciting and very recent work by Onatski \cite{Onatski3}, a new representation was derived for ${}_0 F_0^{(\alpha)}$ in (\ref{eq:Def}), which was expressed as a $r$-dimensional contour integral involving a ${}_0 F_0^{(\alpha)}$ function of only $r \times r$ matrix arguments (rather than $n \times n$ matrix arguments).  This is critical for facilitating the large dimensional analysis of spiked random matrix models, under the asymptotics $n \to \infty$ with $r$ fixed, since the dimension of the matrix arguments does not explode under such asymptotics.  Effectively, the result in \cite{Onatski3} is a generalization of the previous contour-integral formulations in \cite{Forrester,Mo,Wang,Onatski}, which applied for the case $r = 1$. We also mention that, for the same case $r = 1$, an additional generalization of \cite{Forrester,Mo,Wang,Onatski} was presented very recently in \cite{Dharmawansa2}, which provided an analogous contour-integral formula for the case of ${}_p F_q^{(\alpha)}$ with general $p$ and $q$.  


In this paper, to facilitate analysis of LSS of multi-spiked random matrix models, we provide a necessary generalization of the result in \cite{Onatski3} beyond the case $p=0$, $q=0$, by deriving a new  $r$-dimensional contour-integral representation for the hypergeometric function (\ref{eq:Def}), involving a reduced complexity hypergeometric function with $r \times r$ matrix arguments.  This result applies for arbitrary $r$, arbitrary $p$ and $q$, and under some mild conditions on $\alpha$. We keep this discussion general and self-contained, since the class of hypergeometric functions embodied by (\ref{eq:Def}) is very broad, and thus our results may be of independent interest. For the particularization to $\alpha = 1$, which is of prime interest for our LSS analysis, we also derive a new convenient determinant representation, which is general, and allows for arbitrary eigenvalue multiplicities of the rank-$r$ matrix argument.  

Armed with these new results, we can immediately write down new contour-integral-based expressions for the joint eigenvalue densities for each of the three spiked matrix models, with arbitrary numbers of spikes.   Such expressions are in a form which facilitates the analysis of LSS by employing the Coulomb fluid framework from \cite{Chen}, and saddlepoint techniques.  In particular, for each of the three matrix models, we derive CLTs for general linear statistics, presenting rather simple formulas for the asymptotic mean and variance. For the case of Johnstone's spiked model, our expressions provide a simple alternative to those derived previously (using very different methods) in \cite{Wang3}, whilst for the non-central Wishart and $F$ spiked models, the results are completely new.  We find that in all cases, the presence of multiple spikes does not influence the leading order behavior of the asymptotic mean and variance of the LSS, whilst each spike contributes additively to an $O(1)$ correction term to the asymptotic mean, which we characterize explicitly.


{\em Notation.} All columns vectors and matrices are denoted by lowercase and uppercase boldface characters respectively. The conjugate transpose of a matrix $\mathbf{A}$ is $\mathbf{A}^\dag$ whereas its transpose is $\mathbf{A}^T$. $\mathbf{I}_n$ is the identity matrix of size $n \times n$,  whereas $\mathbf{0}_{n\times m}$ is the $n \times m$ matrix of all zeros. $\E(X)$ denotes the expectation of the random variable $X$. $\C\mathcal{W}_n \left (m,\mathbf{\Sigma},\mathbf{\Theta}\right)$ denotes the complex Wishart distribution of size $n$ with $m$ degrees of freedom, scale matrix $\mathbf{\Sigma}$ and non-centrality matrix $\mathbf{\Theta}$. $\mathcal{N}(\mu,\sigma^2)$ denotes the Gaussian distribution with mean $\mu$ and variance $\sigma^2$, whereas $\C\mathcal{N}(\mathbf{u},\mathbf{\Sigma})$ denotes the circularly-symmetric complex Gaussian distribution with mean $\mathbf{u}$ and covariance matrix $\mathbf{\Sigma}$. We use $\overset{\mathcal{L}}{\rightarrow}$ to denote convergence in distribution, and $\mathcal{P}$ to denote Cauchy principal value when dealing with principal value integrals. $a_{1:p}$ denotes a sequence of numbers $a_1, \ldots, a_p$, and similarly,  $a_{1:p} + K$ denotes the sequence $a_1 + K, \ldots, a_p + K$ for any constant $K$. We will also employ $\iota = \sqrt{-1}$.

\section{General Contour-Integral Representation for Hypergeometric Functions of Two Matrix Arguments}



Consider the hypergeometric function of two matrix arguments given in (\ref{eq:Def}),  where $p$ and $q$ are non-negative integers, and $\alpha$ is a real parameter.  Further, denote $x_1, \ldots, x_n$ and $y_1, \ldots, y_n$  as the eigenvalues of the $n \times n$ Hermitian matrices $\mathbf{X}$ and $\mathbf{Y}$, respectively.   Before presenting our main results, we start by recalling some brief background.  First,  recall the classical series expansion
\begin{align}  \label{eq:pFqSeries}
{}_p F_q^{(\alpha)} \left( a_{1:p}; b_{1:q}; \mathbf{X}, \mathbf{Y}   \right) = \sum_{k=0}^\infty \frac{1}{k!} \sum_{\kappa \vdash k, \, \ell (\kappa) \leq n}  \frac{ (a_1)_\kappa^{(\alpha)} \cdots (a_p)_\kappa^{(\alpha)}}{ (b_1)_\kappa^{(\alpha)} \cdots (b_q)_\kappa^{(\alpha)}}  \frac{C_\kappa^{(\alpha)}(\mathbf{X})  C_\kappa^{(\alpha)}(\mathbf{Y}) }{ C_\kappa^{(\alpha)}(\mathbf{I}_n) }
\end{align}
where for all $1\le j \le n$ and $\ell \in \mathbb{Z}^+$ no element of $b_{1:q}$ equals $-\ell+(j-1)/\alpha$, whilst $C_\kappa^{(\alpha)}(\mathbf{X}) = C_\kappa^{(\alpha)}(x_1, \ldots, x_n)$,  $C_\kappa^{(\alpha)}(\mathbf{Y}) = C_\kappa^{(\alpha)}(y_1, \ldots, y_n)$, and $C_\kappa^{(\alpha)}(\mathbf{I}_n) = C_\kappa^{(\alpha)}(1, \ldots, 1)$ are ``$C$-normalized'' Jack polynomials \cite{MacDonald}, which are symmetric homogenous multi-variate polynomials normalized to satisfy:
\begin{align} \label{eq:ZonalDef}
(x_1 + \cdots + x_n)^k = \sum_{\kappa \vdash k, \,  \ell (\kappa) \leq n } C_\kappa^{(\alpha)}(x_1, \ldots, x_n) \; .
\end{align}
The notation $\kappa \vdash k$ means that $\kappa = (\kappa_1, \kappa_2, \ldots)$ is a partition of $k$, with $\ell(\kappa)$ the number of non-zero elements, and the integers $\kappa_1 \geq \kappa_2 \geq \cdots \geq \kappa_{\ell (\kappa)} >  0$ satisfy $\kappa_1 + \kappa_2 + \cdots + \kappa_{\ell(\kappa)} = | \kappa| =  k$. Moreover, $(\cdot)_{\kappa}^{(\alpha)}$ refers to the generalized Pochhammer symbol (see \cite{Dumitriu2007})
\begin{align*} 
(a)_{\kappa}^{(\alpha)} = \prod_{i=1}^{\ell(\kappa)} \frac{ \Gamma \left(   a - \frac{i-1}{\alpha} + \kappa_i  \right) }{  \Gamma \left(  a - \frac{i-1}{\alpha}  \right) }   \; .
\end{align*}
The series \eqref{eq:pFqSeries} converges absolutely for all Hermitian matrices $\mathbf{X}$ and $\mathbf{Y}$ if $p\le q$; for Hermitian matrices $\mathbf{X}$ and $\mathbf{Y}$ satisfying $\max_{ i \le n} |x_i|\times \max_{i \le n} |y_i|<1$ if $p=q+1$, and diverges unless it terminates if $p>q+1$ (see \cite{Gross89}). For all results in this paper, whilst not continually stated,  we  will assume that $\mathbf{X}$ and $\mathbf{Y}$ are chosen such that the series \eqref{eq:pFqSeries} converges.

\begin{remark} Note that since ${}_p F_q^{(\alpha)} \left( a_{1:p}; b_{1:q}; \mathbf{X}, \mathbf{Y}   \right)$ depends on $\mathbf{X}$ and $\mathbf{Y}$ only through their eigenvalues, henceforth we may simply consider $\mathbf{X} = {\rm diag}(x_1, \ldots, x_n)$ and $\mathbf{Y} = {\rm diag}(y_1, \ldots, y_n)$, without loss of generality.  Moreover, as common in the multivariate analysis literature, we have given the definitions above in terms of Hermitian matrices, such that $x_1, \ldots, x_n$ and $y_1, \ldots, y_n$ are real; however, as indicated in \cite{Onatski3}, the definitions above may be equally extended to allow these values to be complex.  
\end{remark}

Note also that for the case where either $p = 0$ or $q = 0$, the corresponding generalized Pochhammer factors on the numerator and denominator of (\ref{eq:pFqSeries}) are omitted; e.g.,
\begin{align*}
{}_0 F_0^{(\alpha)} \left( \mathbf{X}, \mathbf{Y}   \right) = \sum_{k=0}^\infty \frac{1}{k!} \sum_{\kappa \vdash k, \, \ell (\kappa) \leq n }  \frac{C_\kappa^{(\alpha)}(\mathbf{X})  C_\kappa^{(\alpha)}(\mathbf{Y}) }{ C_\kappa^{(\alpha)}(\mathbf{I}_n) } \; .
\end{align*}
Finally, from the definition (\ref{eq:pFqSeries}) and also (\ref{eq:ZonalDef}), for the trivial case $n=1$, 
\begin{align} \label{eq:pFqscalar}
{}_p F_q^{(\alpha)} \left( a_{1:p}; b_{1:q}; x, y   \right) &= \sum_{k=0}^\infty \frac{1}{k!} \frac{ \prod_{j=1}^p   \Gamma(a_j+k) / \Gamma(a_j)}  { \prod_{j=1}^q  \Gamma(b_j+k) / \Gamma(b_j)}  ( x y )^k \nonumber \\
&= {}_p F_q \left( a_{1:p} ; b_{1:q}; x  y   \right) \; ,
\end{align}
which is just the classical ${}_pF_q$ function of one variable\footnote{Scalar hypergeometric functions do not depend on $\alpha$; thus, we will drop the superscript when referring to them.}.

\subsection{New Results}

In this work we are primarily interested in cases where one of the matrix arguments, say $\mathbf{X}$, has rank $r \leq n$, as this is precisely the scenario which arises when dealing with multi-spiked random matrix models, as will be shown in Section \ref{sec:CLT}.  For such cases, the following result provides a new $r$-dimensional contour-integral formula, representing the ${}_pF_q^{(\alpha)}$ function in (\ref{eq:pFqSeries}) of two $n \times n$ matrix arguments in terms of a reduced complexity ${}_pF_q^{(\alpha)}$ function with $r \times r$ matrix arguments.  

\begin{lemma} \label{le:Main}
Assume that $\mathbf{X} = {\rm diag}(x_1, \ldots, x_n)$ and  $\mathbf{Y} = {\rm diag}(y_1, \ldots, y_n)$, with real or complex diagonal entries. Assume also that $x_j \neq 0$ for $1 \leq j \leq r$ and  $x_j = 0$ for $r < j \leq n$.  Denote $\mathbfcal{X} = {\rm diag}\left( x_1, \ldots, x_r  \right)$, and define 
$\mathbfcal{Z} = {\rm diag} \left( z_1, \ldots, z_r \right)$ with $z_j \in \mathbb{C}$.  In addition, let $\alpha = 2/\beta$, with $\beta \in \mathbb{Z}^+$. If $\beta$ is odd, then further assume that $n-r+1$ is even (otherwise, this assumption is not necessary). Define $\theta =  (n-r + 1 - \alpha)/ \alpha$, and assume that  $b_{1:q}$ and $b_{1:q}-\theta$ contain no elements equal to $-\ell+(j-1)/\alpha$ for all $1\le j \le n$ and $\ell \in \mathbb{Z}^+$. Then, we  have 
\begin{align} \label{eq:pFqResult}
& {}_p F_q^{(\alpha)} \left( a_{1:p}; b_{1:q}; \mathbf{X}, \mathbf{Y}   \right) \nonumber \\
& \hspace*{0.3cm} = \frac{\phi^{(\alpha)}(b_{1:q})}{\phi^{(\alpha)}(a_{1:p})}  \frac{1}{r! (2 \pi \iota)^r } \oint_{{\rm C}} \cdots \oint_{{\rm C}}  {}_p F_q^{(\alpha)} \left( a_{1:p} - \theta ; b_{1:q}-\theta ; \mathbfcal{X}, \mathbfcal{Z}  \right) \omega^{(\alpha)}  \left( \mathbfcal{X}, \mathbf{Y}, \mathbfcal{Z} \right) \prod_{j=1}^r {\rm d} z_j
\end{align}
where ${\rm C}$ is a contour, oriented counter-clockwise, that encloses $y_1, \ldots, y_n$. Moreover, 
\begin{align}
\omega^{(\alpha)}  \left( \mathbfcal{X}, \mathbf{Y}, \mathbfcal{Z} \right) &= (-1)^{r(r-1)/(2 \alpha)} \prod_{j=1}^r \left[  \frac{ \Gamma((n+1-j)/ \alpha) \Gamma(1/ \alpha)}{ \Gamma((r+1-j)/ \alpha)  } \right] \nonumber \\
& \hspace*{1cm} \times \prod_{j > i}^r ( z_j - z_i )^{2/\alpha} \prod_{j=1}^r \left[  x_j^{-\theta} \prod_{s=1}^n (z_j - y_s )^{-1/\alpha}  \right] \nonumber
\end{align}
and
\begin{align} \label{eq:phi}
\phi^{(\alpha)} (a_{1:p})  =  \prod_{j=1}^p \left[ \prod_{i=1}^r  \frac{ \Gamma \left( a_j - \frac{i-1}{\alpha} \right)  }{ \Gamma \left(  a_j - \theta - \frac{i-1}{\alpha} \right)   }  \right] \; . 
\end{align}
\end{lemma}

\begin{proof}
See Section \ref{sec:LemProof}.
\end{proof}

\begin{remark}
Note that if $p=0$ or $q=0$, then the corresponding empty products in (\ref{eq:phi}) (equivalently, empty arguments of $\phi^{(\alpha)}$ in (\ref{eq:pFqResult})) are interpreted as unity.  The same is true more generally whenever such empty products are encountered throughout the paper.
\end{remark}

\begin{remark}
The reduction to the case of a single matrix argument; i.e., ${}_p F_q^{(\alpha)} \left( a_{1:p}; b_{1:q}; \mathbf{X}, \mathbf{I}_n \right) = {}_p F_q^{(\alpha)} \left( a_{1:p}; b_{1:q}; \mathbf{X} \right) $ is immediate, upon setting $y_i = 1$ for $i = 1, \ldots, n$.  
\end{remark}

For $p=0$, $q = 0$, Lemma \ref{le:Main} reduces to the result in \cite[Eq.\ (1)]{Onatski3}.   In addition, for the special case $r = 1$, noting (\ref{eq:pFqscalar}), it collapses to the following:
\begin{align*} 
& {}_p F_q^{(\alpha)} \left( a_{1:p}; b_{1:q}; \mathbf{X}, \mathbf{Y}   \right)  \\
& \hspace*{1cm} = \frac{\phi^{(\alpha)}(b_{1:q})}{\phi^{(\alpha)}(a_{1:p})} \frac{1}{ 2 \pi \iota } \oint_{{\rm C}}  {}_p F_q \left( a_{1:p} - n/\alpha + 1; b_{1:q}- n/\alpha + 1; x_1   z \right) \omega^{(\alpha)}  \left( x_1, \mathbf{Y}, z \right)  {\rm d} z
\end{align*}
where now
\begin{align*}
\omega^{(\alpha)}  \left( x_1 , \mathbf{Y}, z \right) = \frac{ \Gamma(n / \alpha)}{    x_1^{n / \alpha - 1} }  \prod_{s=1}^n (z - y_s )^{-1/\alpha}    \; 
\end{align*}
and
\begin{align}
\phi^{(\alpha)} (a_{1:p}) =  \prod_{j=1}^p \frac{ \Gamma(a_j) }{ \Gamma( a_j - n / \alpha + 1) } \; .
\end{align}
This result has been reported very recently in \cite[Proposition 1]{Dharmawansa2}.

\subsubsection{Case of $\alpha = 1$ ($\beta = 2$)}

Whilst the result in Lemma \ref{le:Main} is very general, and we believe may be of independent interest, our main subsequent focus is on the case $\alpha = 1$ ($\beta = 2$), which arises in the study of Hermitian random matrix ensembles.   In this case, we may invoke a determinant formula for the hypergeometric function of two matrix arguments, thus giving a simpler and more convenient representation. These results allow for the $x_i$s to occur with arbitrary multiplicities.


\begin{corollary} \label{corr:Multi}
Consider the same Assumptions as in Lemma \ref{le:Main}, but now also consider $\alpha = 1$ and 
\begin{align}
& x_1 = \cdots = x_{k_1}  \; \;     & =: \tilde{x}_1  \nonumber \\
& x_{k_1 + 1} = \cdots = x_{k_1 + k_2}     & =: \tilde{x}_2  \nonumber \\
& \quad \vdots   &  \vdots \quad \nonumber \\
& x_{ \sum_{\ell=1}^{M-1} k_\ell + 1} = \cdots = x_{r}   & =: \tilde{x}_M \nonumber 
\end{align}
with $\tilde{x}_1 > \cdots > \tilde{x}_M > 0$, where the ordering is imposed without loss of generality.  Then, we have
\begin{align*}
& {}_p F_q^{(1)} \left( a_{1:p}; b_{1:q}; \mathbf{X}, \mathbf{Y}   \right)  =  \frac{K( a_{1:p}, b_{1:q})}{\prod_{ i < j}^M (\tilde{x}_i - \tilde{x}_j )}  \prod_{\ell = 1}^M \frac{ K_{k_\ell} (a_{1:p}, b_{1:q} ) }{ \tilde{x}_\ell^{k_\ell (n-r) }}  {\rm det} \left( \mathbf{A} \right)
\end{align*}
where 
\begin{align} \label{eq:KDef}
K( a_{1:p}, b_{1:q}) = \prod_{\ell = 1}^r  \left[ (n-\ell)!  \prod_{i=1}^q   \frac{ (b_i-\ell)!}{(b_i-n)!}   \prod_{j=1}^p   \frac{ (a_j-n)!}{(a_j-\ell)!}  \right] \; ,
\end{align}
whilst $K_{k_\ell} ( a_{1:p}, b_{1:q}) = 1$ for $k_\ell = 1$ and 
\begin{align*}
K_{k_\ell} ( a_{1:p}, b_{1:q}) =  \prod_{j=1}^{k_\ell - 1} \left[ \frac{1}{(k_\ell - j)!} \frac{ \prod_{i=1}^p (a_i - n + j)! / (a_i - n)! }{ \prod_{i=1}^q (b_i - n + j)! / (b_i - n)! }  \right] \; 
\end{align*}
for $k_\ell > 1$. 
Moreover, $\mathbf{A} = \left[\mathbf{A}_1^T,\dots,\mathbf{A}_M^T\right]^T$ is an $r \times r$ matrix,  
with $k_\ell \times r$ matrix $\mathbf{A}_\ell$  taking entries
\begin{align*}
( \mathbf{A}_\ell )_{i, j} = \frac{1}{2 \pi \iota } \oint_{{\rm C}} \frac{ {}_p F_q \left( a_{1:p} - n+1 + k_\ell - i ; b_{1:q}-n + 1 + k_\ell - i ; x_\ell z   \right) z^{k_\ell - i + j - 1}}{ \prod_{s=1}^n \left( z - y_s \right) } {\rm d} z  \; \; .
\end{align*}
\end{corollary}
\begin{proof}
See Section \ref{sec:Corr1Proof}.
\end{proof}

\begin{corollary}\label{corr:Simple}
If $x_1 > \cdots > x_r$ (i.e., no eigenvalue multiplicities),  then Corollary \ref{corr:Multi} reduces to
\begin{align*} 
& {}_p F_q^{(1)} \left( a_{1:p}; b_{1:q}; \mathbf{X}, \mathbf{Y}   \right)  \\
& \hspace*{0.5cm} =  \frac{K( a_{1:p}, b_{1:q})}{\prod_{ i < j}^r (x_i - x_j ) \prod_{j=1}^r x_j^{n-r}}    {\rm det} \left( \frac{1}{2 \pi \iota } \oint_{{\rm C}} \frac{ {}_p F_q \left( a_{1:p} - n+1 ; b_{1:q}-n + 1 ; x_i z   \right) z^{j-1}}{ \prod_{s=1}^n \left( z - y_s \right) } {\rm d} z \right)_{i,j=1}^r  
\end{align*}
where $K( a_{1:p}, b_{1:q})$ is given by (\ref{eq:KDef}).
\end{corollary}

For the case, $p = 0$, $q = 0$, this agrees with \cite[Corollary 1]{Onatski3}.   Moreover, for the case $r = 1$,  it collapses to
\begin{align} 
& {}_p F_q^{(1)} \left( a_{1:p}; b_{1:q}; \mathbf{X}, \mathbf{Y}   \right)  =  \frac{ K( a_{1:p}, b_{1:q}) }{x_1^{n-1}}    \frac{1}{ 2 \pi \iota } \oint_{{\rm C}} \frac{  {}_p F_q \left( a_{1:p} - n + 1; b_{1:q}- n + 1 ; x_1  z \right)  }{  \prod_{s=1}^n (z - y_s ) }  {\rm d} z \nonumber
\end{align}
where 
\begin{align}
K( a_{1:p}, b_{1:q}) = (n-1)!  \prod_{i=1}^q \frac{ (b_i-1)! }{ (b_i-n)! } \prod_{j=1}^p  \frac{ (a_j-n)! }{ (a_j-1)! }    \;   \nonumber \; .
\end{align}
For $p=0$, $q = 0$, this result agrees with \cite{Forrester,Wang,Mo,Onatski}; for $p=0$, $q = 1$, it agrees with \cite[Eq. (6)]{Dharmawansa}; whilst for $p=1$, $q=1$, it agrees with \cite[Page 18]{Passemier2}.

%

\section{Central Limit Theorems for Linear Spectral Statistics} \label{sec:CLT}

In this section, we derive CLTs for LSS of three ``spiked'' Hermitian random matrix ensembles.  Our results will apply for rather general scenarios, allowing for multiple spikes.  In the following discussion, we will assume that all spikes are \emph{distinct}, which is the most representative scenario for practical applications, such as those highlighted in the introduction (signal processing, biology, finance, etc).  The case where some or all of the spikes coincide will be discussed subsequently, in Section \ref{sec:spikeMult}, for which additional technical difficulties arise.

\subsection{Multi-Spiked Matrix Models and Eigenvalue Distributions}\label{sec:models}
The ``spiked'' complex random matrix models we consider are given as follows:
\begin{list}{$\bullet$}{\leftmargin=2em}
\item {\em Model A:  Spiked central Wishart:} \\ Matrices with distribution $\C\mathcal{W}_n \left (m,\mathbf{\Sigma},\mathbf{0}_{n\times n} \right)$ ($m\ge n$), where $\mathbf{\Sigma}$ has multiple distinct ``spike'' eigenvalues  $1+\delta_1 >  \cdots > 1+\delta_r$, with $\delta_k >0$ for all $1\le k \le r$, and all other eigenvalues equal to $1$. \\ 
\item {\em Model B: Spiked non-central Wishart:}\\
Matrices with distribution $\C\mathcal{W}_n \left (m,\I_n,\mathbf{\Theta} \right)$ ($m\ge n$), where $\mathbf{\Theta}$ is rank $r$ with distinct ``spike'' eigenvalues $n\nu_1 > \cdots > n\nu_r >0$. \\ 
\item {\em Model C: Spiked multivariate F:}\\
Matrices of the form
\[\mathbf{F}=\W_1\W_2^{-1}\text{,}\] where $\W_1 \sim \C\mathcal{W}_n \left (m_1,\mathbf{\Sigma},\mathbf{\Theta}\right)$ ($m_1> n$), $\W_2 \sim \C\mathcal{W}_n \left (m_2,\mathbf{\Sigma},\mathbf{0}_{n\times n} \right)$ ($m_2> n$) are independent, with $\mathbf{\Theta}$ rank $r$ with distinct ``spike'' eigenvalues $n\nu_1 > \cdots > n\nu_r>0$.
\end{list}

For Models A and B, exact expressions for the joint probability density function of the eigenvalues $x_k$, $1\le k \le n$ (taken in the following to be unordered)  are well known, and these are expressed in terms of ${}_0 F_0^{(1)}$ and ${}_0F_1^{(1)}$ functions of two matrix arguments respectively, see \cite[Eq. (95) and (102)]{James}.  Thus, by directly invoking Corollary \ref{corr:Simple}, we immediately obtain new expressions for these eigenvalue distributions, which admit the unified form:
\begin{align}
\frac{ K_{n}^{(A,B)}}{(2 \pi \imath)^r} \prod_{i<j}^n(x_j-x_i)^2 \prod_{j=1}^n  x_j^{m-n}e^{-x_j}     \oint_{\rm C}\cdots\oint_{\rm C} \frac{{\text{det}} \left ( z_j^{i-1}\right )_{i,j=1}^r\text{det}  \left ( l_i(z_j)\right )_{i,j=1}^r}{\prod_{t=1}^r\prod_{s=1}^n (z_t-x_s)}\prod_{j=1}^r\dif z_j \text{,} 
 \label{eq:density}
\end{align}
for $x_k  \in (0, \infty), 1 \leq k \leq n$, the contour ${\rm C}$ encloses counter-clockwise $x_1, \ldots, x_n$ in its interior and $K_{n}^{(A,B)}$ is a normalization constant given explicitly as
\begin{align}
K_{n}^{(A,B)} = \begin{cases}
\frac{\prod_{i < j}^r (\delta_i-\delta_j)^{-1} \prod_{j=1}^r (1+\delta_j)^{r-m-n}}{r!\prod_{k=1}^n (m-k)!\prod_{j=r+1}^n (n-j)!}, &  \text{for Model A} \\
\frac{\prod_{i < j}^r (\nu_i-\nu_j)^{-1} \prod_{j=1}^r \nu_j^{r-n}\exp({-n\sum_{j=1}^r \nu_j})}{r!(m-n)!n^{r(n-r)+1}\prod_{j=r+1}^n (m-j)!(n-j)!},  &  \text{for Model B.}
 \end{cases}\nonumber
\end{align}
The function $l(\cdot)$ captures the effect of the spiked eigenvalues and is given by
\begin{equation}
  l_i(z) = \begin{cases}
    \exp\left ( \frac{\delta_i }{\delta_i+1}z\right), &  \text{for Model A} \\
    ~_0F_1(m-n+1,n \nu_i z),  &  \text{for Model B.}
  \end{cases}\nonumber
\end{equation}


Similarly, for Model C, an exact expression for the joint probability density function of the eigenvalues $x_k$,  $1 \leq k \leq n$ (taken in the following to be unordered) is well known, and this is given in terms of a ${}_1F_1^{(1)}$ functions of two matrix arguments, see \cite[Eq. (109)]{James}. Thus, by invoking Corollary \ref{corr:Simple} and applying a change of variable, we immediately obtain a new expression for the joint density of $\mathsf{f}_k=x_k/(1+x_k) \in (0,1)$:
\begin{align}
&\frac{ K_{n}^{(C)}}{(2 \pi \imath)^r}\prod_{i < j}^n (\mathsf{f}_j-\mathsf{f}_i)^2
 \prod_{j=1}^n \mathsf{f}_j^{m_1-n}(1-\mathsf{f}_j)^{m_2-n} \oint_{\rm C}\cdots\oint_{\rm C} \frac{{\text{det}} \left ( z_j^{i-1}\right )_{i,j=1}^r\text{det}  \left ( l_i(z_j)\right )_{i,j=1}^r }{\prod_{t=1}^r\prod_{s=1}^n (z_t-\mathsf{f}_s)}\prod_{j=1}^r\dif z_j \text{,}\label{eq:densityF}
\end{align}
where the contour ${\rm C}$ encloses counter-clockwise $\mathsf{f}_1, \ldots, \mathsf{f}_n$ in its interior, $K_{n}^{(C)}$ is a normalization constant given explicitly as
\begin{align*}
K_{n}^{(C)}=\frac{\prod_{i < j}^r (\nu_i-\nu_j)^{-1} \prod_{j=1}^r \nu_j^{r-n}\e^{-n\sum_{j=1}^r \nu_j}(m_1+m_2-n)!\prod_{j=r+1}^n (m_1+m_2-j)!}{r!n^{r(n-r)+1}\prod_{j=r+1}^n (n-j)!\prod_{k=1}^n (m_2-k)!(m_1-n)!\prod_{j=r+1}^n (m_1-j)!}\text{,}
\end{align*}
whilst the function $l(\cdot)$ captures the effect of the spiked eigenvalues and is given by
\begin{align*}  l_i(z) =  {}_1 F_1( m_1+m_2-n+1, m_1-n+1; n\nu_i z ) \text{.}
\end{align*} 

In the following we will compute the asymptotic distribution of general LSS for each of the three multi-spiked random matrix models above. In taking asymptotics for Models A and B, we will be concerned with the following limits:
\begin{assumption}\label{assumptionAB}
$m,n \rightarrow \infty$ such that $m/n \rightarrow c \geq 1$.
\end{assumption}
For Model C, we will be concerned with:
\begin{assumption}\label{assumptionC}
$m_1,m_2,n \rightarrow \infty$ such that $m_1/n \rightarrow c_1 > 1$ and $m_2/n \rightarrow c_2 > 1$.
\end{assumption}

\subsection{Central Limit Theorems}\label{sec:mainresults}
\begin{theorem}\label{th:wishart}
Consider Models A and B and define
\begin{align*}
a=(1-\sqrt{c})^2, \quad b=(1+\sqrt{c})^2 \text{.}  
\end{align*}
Under Assumption \ref{assumptionAB}, for an analytic function $f:\mathcal{U}\mapsto \C$ where $\mathcal{U}$ is an open subset of the complex plane which contains $[a,b]$, we have
\begin{align*}
\sum_{k=1}^n f \left ( \frac{x_k}{n} \right )-n\mu \; \overset{\mathcal{L}}{\rightarrow} \; \mathcal{N}\left ( \sum_{\ell=1}^r\bar{\mu}(z_{0,\ell}), \sigma^2 \right )\text{,}  
\end{align*}
where
\begin{align}
\mu &=\frac{1}{2\pi}\int_a^b f(x) \frac{\sqrt{(b-x)(x-a)}}{x}\, \dif x  \label{eq:mu} \\
\sigma^2 &=\frac{1}{2\pi^2}\int_a^b \frac{f(x)}{\sqrt{(b-x)(x-a)}} \left [ \mathcal{P} \int_a^b \frac{f'(y)\sqrt{(b-y)(y-a)}}{x-y} \,\dif y \right ]\, \dif x \label{eq:V1} 
\end{align}
with these terms independent of the spikes. The spike-dependent terms $\bar{\mu}(z_{0,\ell})$, $1\le \ell \le r$ admit
\begin{align}
\bar{\mu}(z_{0,\ell})&=\frac{1}{2\pi} \int_a^b \frac{f(x)}{\sqrt{(b-x)(x-a)}}\left (\frac{\sqrt{(z_{0,\ell}-a)(z_{0,\ell}-b)}}{z_{0,\ell}-x}-1 \right )  \, \dif x    \label{eq:muVal}
\end{align}
where
\begin{align*}
z_{0,\ell}=\begin{cases}
\frac{(1+c\delta_\ell)(1+\delta_\ell)}{\delta_\ell}, &  \text{for Model A} \\
\frac{(1+\nu_\ell)(c+\nu_\ell)}{\nu_\ell}, &  \text{for Model B}
\end{cases} \; .
\end{align*}
The branch of the square root $\sqrt{(z_{0,\ell}-a)(z_{0,\ell}-b)}$ is chosen according to Remark \ref{rem:smodelAB}. 
\end{theorem}
\begin{proof}
See Section \ref{sec:proofwishart}.
\end{proof}

\begin{theorem} \label{th:F}
Consider Model C and define
\begin{align*}
a &= \frac{c_1(c_1+c_2-1)+c_2-2\sqrt{c_1c_2(c_1+c_2-1)}}{(c_1+c_2)^2}\text{,}  \\
 b &= \frac{c_1(c_1+c_2-1)+c_2+2\sqrt{c_1c_2(c_1+c_2-1)}}{(c_1+c_2)^2} \text{.}  
\end{align*}
Under Assumption \ref{assumptionC}, for an analytic function $f:\mathcal{U}\mapsto \C$ where $\mathcal{U}$ is an open subset of the complex plane which contains $[a,b]$, we have
\begin{align*} 
\sum_{k=1}^n f \left ( x_k \right )-n\mu_\F \; \overset{\mathcal{L}}{\rightarrow} \; \mathcal{N}\left ( \sum_{\ell=1}^r \bar{\mu}_\F(z_{0,\ell}), \sigma_\F^2 \right )
\end{align*}
where
\begin{align}
\mu_\F &=\frac{c_1+c_2}{2\pi}\int_a^b f\left(\frac{x}{1-x}\right) \frac{\sqrt{(b-x)(x-a)}}{x(1-x)}\, \dif x   \label{eq:muF} \\
\sigma_\F^2 &=\frac{1}{2\pi^2}\int_a^b \frac{f\left(\frac{x}{1-x}\right)}{\sqrt{(b-x)(x-a)}} \left [ \mathcal{P} \int_a^b \frac{f'\left(\frac{y}{1-y}\right)\sqrt{(b-y)(y-a)}}{x-y} \,\dif y \right ]\, \dif x  \; . \label{eq:V1F}  
\end{align}
The spike-dependent terms $\bar{\mu}_\F(z_{0,\ell})$, $1\le \ell \le r$ admit
\begin{align}
\bar{\mu}_\F(z_{0,\ell})&=\frac{1}{2\pi} \int_a^b \frac{f\left(\frac{x}{1-x}\right)}{\sqrt{(b-x)(x-a)}}\left (\frac{\sqrt{(z_{0,\ell}-a)(z_{0,\ell}-b)}}{z_{0,\ell}-x}-1 \right )  \, \dif x    \label{eq:muValF}
\end{align}
where
\begin{align*}
z_{0,\ell} &=\frac{(1+\nu_\ell)(c_1+\nu_\ell)}{\nu_\ell(c_1+c_2+\nu_\ell)} \text{.}
\end{align*}
The branch of the square root $\sqrt{(z_{0,\ell}-a)(z_{0,\ell}-b)}$ is chosen according to Remark \ref{rem:smodelC}.
\end{theorem}
\begin{proof}
See Section \ref{sec:SaddleC}.
\end{proof}

For all three models, these theorems generalize our previous results in \cite{Passemier2} to multi-spiked scenarios, collapsing to the same expressions for the special case $r = 1$.  We see that each spiked eigenvalue, whilst not affecting the asymptotic mean or variance to leading order, contributes an $O(1)$ correction term to the mean, and the contributions across multiple spikes are additive.

\subsection{Extension to spike multiplicities} \label{sec:spikeMult}

For completeness, we now consider the case where some or all of the spiked eigenvalues coincide.  To this end, we require a slight reformulation of the three matrix models introduced previously, with additional notation:
\begin{list}{$\bullet$}{\leftmargin=2em}
\item {\em Model A:  Spiked central Wishart with spike multiplicities:} \\ As before, but now $\mathbf{\Sigma}$ has $M$ distinct spike eigenvalues  $1+\delta_1 >  \cdots > 1+\delta_M$, with $\delta_\ell >0$ having multiplicity $k_\ell$, for $1\le \ell \le M$, such that $\sum_{\ell=1}^M k_\ell = r$, and all other eigenvalues equal to $1$. \\ 
\item {\em Model B: Spiked non-central Wishart with spike multiplicities:}\\
As before, but now the rank-$r$ matrix $\mathbf{\Theta}$ has $M$ distinct spike eigenvalues $n\nu_1 > \cdots > n\nu_M >0$, where $\nu_\ell$ has multiplicity $k_\ell$, for $1\le \ell \le M$, such that $\sum_{\ell=1}^M k_\ell = r$. \\ 
\item {\em Model C: Spiked multivariate F with spike multiplicities:}\\
As before, but now the rank-$r$ matrix $\mathbf{\Theta}$ has $M$ distinct spike eigenvalues $n\nu_1 > \cdots > n\nu_r>0$, where $\nu_\ell$ has multiplicity $k_\ell$, for $1\le \ell \le M$, such that $\sum_{\ell=1}^M k_\ell = r$.
\end{list}
With spike multiplicities, whilst one may naively set the spikes to be equal in Theorems \ref{th:wishart} and \ref{th:F}, this is not mathematically justified since the eigenvalue densities given in \eqref{eq:density} and \eqref{eq:densityF}, used in deriving these theorems, are no longer valid.  This is because $\text{det}(l_i(z_j))_{i,j=1}^r$ equates to $0$ whilst the constants  $K_{n}^{(A,B)}$ and $K_{n}^{(C)}$ tend to infinity.   Therefore, here we consider the appropriate modifications which take account of spike multiplicities.  As we will see, it turns out that deriving CLTs for LSS of the three matrix models, analogous to Theorems \ref{th:wishart} and \ref{th:F}, is faced with additional technical challenges when the spikes coincide.  Thus, as we will describe, this will prevent us from deriving the desired CLT in full generality.

We start by presenting the appropriately modified joint eigenvalue densities for each of the three spiked matrix models.  These are obtained by replacing the hypergeometric functions of two matrix arguments in \cite[Eq. (95), (102) and (109)]{James}, as before, but now we use Corollary \ref{corr:Multi} rather than the simpler expression in Corollary \ref{corr:Simple}. Consequently, the joint probability density function of the eigenvalues $x_k$, $1 \le k \le n$ for Models A and B admits the unified form
\begin{align}
\frac{ K_{n}^{(A,B)}}{(2 \pi \imath)^r} \prod_{i<j}^n(x_j-x_i)^2 \prod_{j=1}^n  x_j^{m-n}e^{-x_j}     \oint_{\rm C}\cdots\oint_{\rm C} \frac{{\text{det}} \left ( z_j^{i-1}\right )_{i,j=1}^r\text{det}(\bar{\mathbf{A}} )}{\prod_{t=1}^r\prod_{s=1}^n (z_t-x_s)}\prod_{j=1}^r\dif z_j \text{,} 
 \label{eq:densityMult}
\end{align}
for $x_k  \in (0, \infty), 1 \leq k \leq n$, the contour ${\rm C}$ encloses counter-clockwise $x_1, \ldots, x_n$ in its interior and $K_{n}^{(A,B)}$ is a normalization constant given explicitly as
\begin{align}
K_{n}^{(A,B)}= \begin{cases}
\frac{\prod_{i < j}^M (\delta_i-\delta_j)^{-1} \prod_{j=1}^M (1+\delta_j)^{k_\ell(r-m-n)}}{r!\prod_{k=1}^n (m-k)!\prod_{j=r+1}^n (n-j)!}, &  \text{for Model A} \\
\frac{\prod_{i < j}^M (\nu_i-\nu_j)^{-1} \prod_{j=1}^M \nu_j^{k_\ell(r-n)}\exp({-n \sum_{j=1}^M k_\ell \nu_j})}{r!n^{r(n-r)+1}\prod_{j=r+1}^n (m-j)!(n-j)!\prod_{\ell=1}^M\prod_{j=1}^{k_\ell-1}(k_\ell-j)!(m-n+j)!},  &  \text{for Model B.}
 \end{cases}  \nonumber
\end{align}
Moreover, $\bar{\mathbf{A}} = \left[\bar{\mathbf{A}}_1^T,\dots,\bar{\mathbf{A}}_M^T\right]^T$ is an $r \times r$ matrix which captures the effect of the spiked eigenvalues, with $k_\ell \times r$ submatrix $\bar{\mathbf{A}}_\ell$ taking entries 
\begin{align*}
( \bar{\mathbf{A}}_\ell )_{i,j} =\begin{cases} 
z_j^{k_\ell - i} \exp \left ( \frac{\delta_\ell}{\delta_\ell+1} z_j \right ) & \text{for Model A}\\
z_j^{k_\ell - i} {}_0 F_1 ( m-n+k_\ell-i+1; n\nu_\ell z_j )&\text{for Model B.}
\end{cases}
\end{align*}

Similarly, with $x_k$, $1 \le k \le n$ the eigenvalues for Model C, we obtain the joint density of $\mathsf{f}_k=x_k/(1+x_k) \in (0,1)$:
\begin{align*}
&\frac{ K_{n}^{(C)}}{(2 \pi \imath)^r}\prod_{i < j}^n (\mathsf{f}_j-\mathsf{f}_i)^2
 \prod_{j=1}^n \mathsf{f}_j^{m_1-n}(1-\mathsf{f}_j)^{m_2-n} \oint_{\rm C}\cdots\oint_{\rm C} \frac{{\text{det}} \left ( z_j^{i-1}\right )_{i,j=1}^r\text{det}  (\bar{\mathbf{A}}) }{\prod_{t=1}^r\prod_{s=1}^n (z_t-\mathsf{f}_s)}\prod_{j=1}^r\dif z_j \text{,}
\end{align*}
where the contour ${\rm C}$ encloses counter-clockwise $\mathsf{f}_1, \ldots, \mathsf{f}_n$ in its interior, $K_{n}^{(C)}$ is a normalization constant given explicitly as
\begin{align*}
&K_{n}^{(C)}=\prod_{\ell=1}^M\prod_{j=1}^{k_\ell-1}\frac{(m_1+m_2-n+j)!}{(k_\ell-j)!(m_1-n+j)!}\\
& \hspace*{1cm}\times \frac{\prod_{i < j}^M (\nu_i-\nu_j)^{-1} \prod_{j=1}^M \nu_j^{k_\ell(r-n)}\e^{-n\sum_{j=1}^M k_\ell\nu_j}\prod_{j=r+1}^n (m_1+m_2-j)!}{r!n^{r(n-r)+1}\prod_{j=r+1}^n (n-j)!\prod_{k=1}^n (m_2-k)!\prod_{j=r+1}^n (m_1-j)!} 
\text{,}
\end{align*}
whilst $\bar{\mathbf{A}} = \left[\bar{\mathbf{A}}_1^T,\dots,\bar{\mathbf{A}}_M^T\right]^T$ is an $r \times r$ matrix which captures the effect of the spiked eigenvalues, with $k_\ell \times r$ submatrix $\bar{\mathbf{A}}_\ell$ taking entries 
\begin{align*}
( \bar{\mathbf{A}}_\ell )_{i,j} =
z_j^{k_\ell - i} {}_1 F_1 ( m_1+m_2-n+k_\ell-i+1,m_1-n+k_\ell-i+1; n\nu_\ell z_j )\text{.}
\end{align*}

Whilst facilitating the derivations of the desired CLTs, these joint eigenvalue density expressions may also be of independent interest. With these, we can now tread the footsteps of the proof of Theorem \ref{th:wishart} for Models A and B, and that of Theorem \ref{th:F} for Model C.  It turns out that the derivation follows in a straightforward manner, up to the point of computing the saddlepoint approximation, in which case we encounter difficulties.  In particular, it turns out that as the number of spike multiplicities $k_\ell$ increase, one requires a more and more accurate saddlepoint approximation, which becomes unweildy beyond small multiplicity scenarios.  In Section \ref{sec:ext} we give specific details of the derivation method, where we employ a refined saddlepoint approximation (i.e., by including a correction term) which allows computation of the desired CLT for the cases $r = 2$ and $r = 3$, with any spike multiplicity. We also show that with $r=4$, a further refinement of the saddlepoint approximation is needed. We give these derivations explicitly for Models A and B, but the corresponding results follow readily for Model C also.  For $r=2$ and $r=3$, we establish the same result as in Theorem \ref{th:wishart} and Theorem \ref{th:F}, upon setting the spikes to be equal.  These observations motivate the following claim:

{\bf Claim 1:} \emph{In addition to the case of distinct spikes, the results in Theorems \ref{th:wishart} and \ref{th:F} continue to hold when some or all of the spikes coincide.}

Whilst it appears difficult to prove this claim in full generality with our current derivation methods, the evidence provided for the cases $r = 2$ and $r = 3$, as well as all of our numerical simulations for $r \geq 4$ suggest that this is indeed true.  We point out that this claim is consistent with results reported in \cite{Wang3}, which considered Model A (but not Models B and C) and accounted for arbitrary multiplicities. Therein, as indicated previously, they used substantially different methods to give an alternative to Theorem \ref{th:wishart}, and they demonstrated that it was inconsequential to their final result whether or not the spikes were distinct or equal.

\section{Derivations of main results}

This section compiles the proofs of the key technical results in the paper.

\subsection{Proof of Lemma \ref{le:Main}} \label{sec:LemProof}

\begin{proof}

For the most part, the proof mirrors that of \cite[Appendix A]{Onatski3}, which considered the case $p = 0$, $q = 0$.  We first give a slight modification of the result in \cite[Lemma G.1]{Onatski3}, which presents the  torus scalar product relation for Jack polynomials. In particular, letting $\mathbfcal{Z} = {\rm diag}(z_1, \ldots, z_r)$ and $\mathbfcal{Z}^* = {\rm diag}(z^*_1, \ldots, z^*_r)$, with ${}^*$ denoting complex conjugate, we have
\begin{align} \label{eq:Torus}
\big \langle  C_{\kappa}^{(\alpha)}  ( \mathbfcal{Z} ), C_\tau^{(\alpha)} ( \mathbfcal{Z} )  \big \rangle_{\alpha} &:= \frac{1}{r! (2 \pi \iota)^r}  \oint \cdots \oint C_\kappa^{(\alpha)} ( \mathbfcal{Z} )  C_\tau^{(\alpha)} ( \mathbfcal{Z}^* )  \frac{ \prod_{1 \leq i,j \leq r , \;  i \neq j} \left( 1 - z_i / z_j \right)^{\frac{1}{\alpha}} }{ \prod_{j=1}^r  z_j}  \prod_{j=1}^r {\rm d} z_j  \nonumber \\
&= \left\{  
\begin{array}{ll}
         \left( \alpha^{|\kappa|} | \kappa |! \right)^2 \left( \prod_{j=1}^r \frac{ \Gamma( (r-j+1)/\alpha) }{ \Gamma(1/ \alpha) \Gamma( 1 + (r-j)/\alpha) } \right) \frac{ v (\kappa, \alpha) }{ w ( \kappa, \alpha) } & \mbox{if $\kappa = \tau$}  \\
        0 & \mbox{if $\kappa \neq \tau$}
\end{array}
\right.
\end{align}
where the integration contours are unit circles in the complex plane, whilst $v (\kappa, \alpha)$ and $w (\kappa, \alpha)$ are constants which are related to the so-called Ferrers diagram (equivalently, Young diagram) of the partition $\kappa$.  The specific values of these constants will not be needed here; but for specific details, refer to \cite[Appendix A]{Onatski3} and \cite{Dumitriu2007}.  Recall also the notation $| \kappa | = \kappa_1 + \kappa_2 + \cdots$. 

Consider the right-hand side of (\ref{eq:pFqResult}), which we label ${\rm RHS}$.  One can verify that this quantity remains unchanged upon transforming
\begin{align*}
\mathbfcal{Z}\to \psi \mathbfcal{Z}, \quad  \mathbf{Y} \to \psi \mathbf{Y},  \quad \mathbf{X} \to \psi^{-1} \mathbf{X} , 
\end{align*}
for any constant $\psi$. Thus, without loss of generality, we may henceforth assume that ${\rm max}_{j\le n} \, |y_j|  < 1$.
\begin{remark}\label{rem:maxx}
When ${\rm max}_{j\le n} \, |x_j| \times {\rm max}_{j\le n} \, |y_j|  < 1$, we can additionally assume that ${\rm max}_{j \le n} \, |x_j|  < 1$. Indeed, define
\begin{align*}
\psi=\frac{1}{\displaystyle\max_{j \le n} \, |\psi y_j|+\varepsilon}\text{,} \quad\quad\text{with }0 <\varepsilon < \frac{1}{\displaystyle\max_{j \le n} |x_j|} - \max_{j \le n} |y_j|\text{.}
\end{align*}
With this choice of $\psi$, one can verify that ${\rm max}_{j\le n} \, |\psi y_j|  < 1$ and ${\rm max}_{j\le n} \, |\psi^{-1} x_j|  < 1$.
\end{remark}
Moreover, since each $z_j$ defined in ${\rm RHS}$ traces a contour which encircles $\{ y_j \}_{j=1}^n$, due to the scaling above, we may now deform each of these contours to trace the unit circle in the complex plane, without changing the integral.  This will allow us to make use of (\ref{eq:Torus}).
 
To commence the proof, we expand the ${}_p F_q^{(\alpha)}$ function in ${\rm RHS}$ via (\ref{eq:pFqSeries}). When $p=q+1$, this is possible since the maximum eigenvalue of $\mathbfcal{X}$ is $\max_{ j \le n} |x_j|$ which is less than one (see Remark \ref{rem:maxx}) and $\max_{ j \le n} |z_j|=1$, so that the series converges. Since $\prod_{j>i}^r(z_j-z_i)^{2/\alpha}$ equals $(-1)^{r(r-1)/(2\alpha)}\prod_{j=1}^r z_j^{(r-1)/\alpha}\prod_{1 \leq i,j \leq r , \;  i \neq j} (1-z_iz_j^{-1})^{1/\alpha}$, we rewrite $\omega^{(\alpha)}$ as
\begin{align} \label{eq:Om}
\omega^{(\alpha)}  \left( \mathbfcal{X}, \mathbf{Y}, \mathbfcal{Z} \right) = \gamma^{(\alpha)}   \prod_{j=1}^r (z_ j x_j)^{- \theta}   \left[ \prod_{j=1}^r   \prod_{s = 1}^n \left( 1 - y_s z_j^* \right)   \right]  \frac{ \prod_{1 \leq i,j \leq r , \;  i \neq j} \left( 1 - z_i / z_j \right)^{1/\alpha} }{ { \prod_{j=1}^r  z_j} }
\end{align}
where
\begin{align*}
\gamma^{(\alpha)} = \prod_{j=1}^r \left[ \frac{ \Gamma( (n+1-j)/\alpha ) \Gamma(1/\alpha) }{ \Gamma( (r+1-j)/\alpha) }    \right] \;  .
\end{align*}
In addition, the quantity in square brackets in (\ref{eq:Om}) admits the expansion \cite[Lemma G.2]{Onatski3}
\begin{align*}
\prod_{j=1}^r  \prod_{s = 1}^n \left( 1 - y_s z_j^* \right)   = \sum_{t=0}^\infty  \sum_{\tau \vdash t, \, \ell(\tau) \leq r} \frac{ w( \tau, \alpha )}{ (\alpha^{|\tau|} |\tau|! )^2 } C_\tau^{(\alpha)} (\mathbf{Y})  C_\tau^{(\alpha)} (\mathbfcal{Z}^*) \; .
\end{align*}
As explained in \cite{Onatski3}, the series on the right-hand side of this equality converges uniformly over $\Omega_\rho=\{\mathbfcal{Z}^*: \max_{j\le r} |z_j^*| \le \rho^{-1}\}$, for any $\rho > \max_{s\le n} |y_s|$ since the function on the left-hand side is analytic in a open region containing $\Omega_\rho$. With these, the right-hand side of (\ref{eq:pFqResult}) becomes
\begin{align} \label{eq:RHS}
{\rm RHS} &= \frac{\phi^{(\alpha)}(b_{1:q})}{\phi^{(\alpha)}(a_{1:p})} \frac{\gamma^{(\alpha)}}{ \prod_{j=1}^r x_j^{\theta} }  \sum_{k=0}^\infty \, \sum_{\kappa \vdash k, \, \ell(\kappa) \leq r } \, \sum_{t=0}^\infty \, \sum_{\tau \vdash t, \, \ell( \tau ) \leq r } \frac{1}{k!} \frac{ (a_{1:p} - \theta)_\kappa^{(\alpha)} }{ (b_{1:q} - \theta)_\kappa^{(\alpha)} } \frac{ w( \tau, \alpha )}{ (t! \alpha^t )^2 }  \frac{C_\kappa^{(\alpha)}(\mathbfcal{X})  C_\tau^{(\alpha)}(\mathbf{Y}) }{ C_\kappa^{(\alpha)}(\mathbf{I}_r) } \nonumber \\
& \hspace*{3cm} \times \Big \langle  C_{\kappa}^{(\alpha)} (\mathbfcal{Z}), \biggl[ \prod_{j=1}^r z_j^{\theta} \biggr] C_\tau^{(\alpha)} (\mathbfcal{Z})   \Big \rangle_{\alpha}
\end{align}
where, for notational convenience, we have introduced 
\begin{align} \label{eq:PochNew}
(a_{1:p} - \theta)_\kappa^{(\alpha)} = \prod_{j=1}^p  (a_j - \theta)_\kappa^{(\alpha)}  , \quad \quad \quad (b_{1:q} - \theta)_\kappa^{(\alpha)} = \prod_{j=1}^q (b_j - \theta)_\kappa^{(\alpha)}  \; .
\end{align}
The interchange of the order of integration and summation in \eqref{eq:RHS} is possible because, as seen previously, the two series converge uniformly over the unit torus.

To proceed, given that $\theta$ is an integer (enforced in the lemma statement by requiring $n-r+1$ to be even whenever $\beta$ is odd), we may utilize the recurrence relation for the ``$J$-normalized'' Jack polynomials in \cite[Proposition 5.1]{Stanley}, along with their connection to our ``$C$-normalized'' Jack polynomials \cite[Table 6]{Dumitriu2007}, to obtain\footnote{Hereby, we correct an argument of \cite{Onatski3}, where the multiplying factor $K_\tau (\alpha)$ was assumed unity.  It turns out, however, that this does not affect the final result.}
\begin{align}  \label{eq:NewProp}
\left[ \prod_{j=1}^r z_j^{\theta} \right] C_\tau^{(\alpha)} (\mathbfcal{Z}) =  K_\tau (\alpha) C_{\tilde{\tau}}^{(\alpha)} (\mathbfcal{Z}) \; ,
\end{align}
where
\begin{align} \label{eq:KConst}
K_\tau (\alpha) = \alpha^{- r \theta}  \frac{t!}{(t+r \theta)!}  \frac{ w( \tilde{\tau}, \alpha ) }{ w( \tau, \alpha ) }  \prod_{j=1}^\theta \frac{1}{ c_{(\tau_1 + j, \cdots, \tau_r + j)}  (\alpha) }
\end{align}
and
\begin{align} \label{eq:tautild}
\tilde{\tau} = ( \tau_1 + \theta , \ldots, \tau_r + \theta ) \text{.}
\end{align}
Here, $\tilde{\tau}$ is a partition corresponding to $\tau$ but with each element shifted by a constant $\theta$, whilst the constants $c_{(\tau_1 + j, \cdots, \tau_r + j)} (\alpha)$, $j = 1, \ldots, \theta$, are defined based on the Ferrers diagram of the associated partitions; the specific values of these are given in \cite[Proposition 5.5]{Stanley}, though this will not be needed subsequently.


\begin{remark} 

For the special case $\alpha = 1$ ($\beta = 2$), we can obtain a simple and explicit expression for (\ref{eq:KConst}).  In particular, following \cite{James},  
\begin{align*}
C_\tau^{(1)} (\mathbfcal{Z}) = \chi_{[\tau]} (1) \chi_{ \{ \tau \} } (\mathbfcal{Z})
\end{align*}
where $\chi_{[\tau]} (1)$ and $\chi_{ \{ \tau \} } (\mathbfcal{Z})$ are representation-theoretic quantities given by
\begin{align} 
\chi_{[\tau]} (1) = t! \frac{ \prod_{i<j}^r (\tau_i - \tau_j - i + j) }{ \prod_{j=1}^r  ( r + \tau_j - j )! }  , \quad \quad  \chi_{ \{ \tau \} } (\mathbfcal{Z}) = \frac{  {\rm det}\left( z_i^{\tau_j + r - j} \right)_{i,j=1}^r } { {\rm det}\left( z_i^{r - j} \right)_{i,j=1}^r }  \nonumber 
\end{align}
from which it becomes clear that 
\begin{align*}
K_\tau (1) = \frac{ \chi_{[ \tau ]}(1) }{ \chi_{[ \tilde{\tau} ]}(1) } \, = \,  \frac{t!}{(t+r(n-r))!} \prod_{j=1}^r \frac{ (n+ \tau_j - j ) ! }{ (r + \tau_j - j ) ! }   \; .
\end{align*}

\end{remark}

With the above results, the second line of (\ref{eq:RHS}) can be written as
\begin{align*}
\Big \langle  C_{\kappa}^{(\alpha)} (\mathbfcal{Z}), \biggl[ \prod_{j=1}^r  z_j^\theta \biggr] C_\tau^{(\alpha)} (\mathbfcal{Z})   \Big \rangle_{\alpha} = K_\tau (\alpha)  \Big \langle  C_{\kappa}^{(\alpha)} (\mathbfcal{Z}),  C_{\tilde{\tau}}^{(\alpha)} (\mathbfcal{Z})   \Big \rangle_{\alpha}  \; .
\end{align*}
Now, from (\ref{eq:Torus}), observe that this is zero whenever $\kappa \neq \tilde{\tau}$.  From the expression (\ref{eq:RHS}), for every $\tilde{\tau}$ there is a matching $\kappa$ (but not vice-versa, since, from (\ref{eq:tautild}), each of the elements of partition $\tilde{\tau}$ are constrained to be greater than or equal to $\theta$, whilst there is no such constraint for $\kappa$).  For this reason, (\ref{eq:RHS}) simplifies to 
\begin{align*}
{\rm RHS}  &= \frac{\phi^{(\alpha)}(b_{1:q})}{\phi^{(\alpha)}(a_{1:p})} \frac{\gamma^{(\alpha)}}{ \prod_{j=1}^r x_j^\theta } \sum_{t=0}^\infty  \sum_{\tau \vdash t, \, \ell (\tau) \leq r} \frac{1}{(t + r \theta)!} \frac{ (a_{1:p} - \theta )_{\tilde{\tau}}^{(\alpha)} }{ (b_{1:q} - \theta)_{\tilde{\tau}}^{(\alpha)} } \frac{ w( \tau, \alpha )}{ (t! \, \alpha^t )^2 }  \frac{C_{\tilde{\tau}}^{(\alpha)}(\mathbfcal{X})  C_\tau^{(\alpha)}(\mathbf{Y}) }{ C_{\tilde{\tau}}^{(\alpha)}(\mathbf{I}_r) } \\
& \hspace*{3cm} \times K_\tau (\alpha)    \Big \langle  C_{\tilde{\tau}}^{(\alpha)}(\mathbfcal{Z}),  C_{\tilde{\tau}}^{(\alpha)} (\mathbfcal{Z})   \Big \rangle_{\alpha}  \; .
\end{align*}
Applying now (\ref{eq:Torus}), we evaluate the torus scalar product.  Similar to (\ref{eq:NewProp}), we may also write
\begin{align*}
 C_{\tilde{\tau}}^{(\alpha)} (\mathbfcal{X}) =   \frac{ \prod_{j=1}^r   x_j^{\theta} }{K_\tau (\alpha) }  \,   C_{\tau}^{(\alpha)} (\mathbfcal{X}) \; 
\end{align*}
whilst, in addition, \cite[Lemma 7]{Dubbs} implies that $C_{\tau}^{(\alpha)} (\mathbfcal{X}) = C_{\tau}^{(\alpha)} (\mathbf{X})$ for any $\tau$ with $\ell(\tau) \leq r$, and $C_{\tau}^{(\alpha)} (\mathbf{X}) = 0$ if $\ell(\tau) > r$.  Together, these results yield
%
\begin{align} \label{eq:RHSLast}
{\rm RHS}  &= \frac{\phi^{(\alpha)}(b_{1:q})}{\phi^{(\alpha)}(a_{1:p})} \sum_{t=0}^\infty  \sum_{\tau \vdash t, \, \ell( \tau ) \leq n} \tilde{\gamma}^{(\alpha)}  \frac{1}{t!} \frac{ (a_{1:p} - \theta)_{\tilde{\tau}}^{(\alpha)} }{ (b_{1:q} - \theta)_{\tilde{\tau}}^{(\alpha)} }  \frac{C_{\tau}^{(\alpha)}(\mathbf{X})  C_\tau^{(\alpha)}(\mathbf{Y}) }{ C_{\tau}^{(\alpha)}(\mathbf{I}_n) } 
\end{align}
where
\begin{align*}
 \tilde{\gamma}^{(\alpha)} = \frac{ \alpha^{2r \theta } ( t + r \theta )! w( \tau, \alpha)  v(\tilde{\tau}, \alpha )  }{ t! \; w(\tilde{\tau}, \alpha)  } \frac{ C_\tau^{(\alpha)} ( \mathbf{I}_n ) }{ C_{\tilde{\tau}}^{(\alpha)} ( \mathbf{I}_r ) }  \prod_{j=1}^r \frac{ \Gamma( (n+1-j)/\alpha  )  }{  \Gamma( 1 + (r-j)/\alpha  )  } \; .
\end{align*}
This turns out to be the same factor obtained in \cite[Page 22]{Onatski3}, where it was evaluated as
\begin{align*}
\tilde{\gamma}^{(\alpha)} = 1 \; .
\end{align*}
Finally, recalling (\ref{eq:PochNew}), we note that 
\begin{align}
(a_{1:p} - \theta )_{\tilde{\tau}}^{(\alpha)} &=  \prod_{j=1}^p  (a_j - \theta)_{\tilde{\tau}}^{(\alpha)}  \nonumber \\ 
&= \prod_{j=1}^p \prod_{i=1}^r  \frac{ \Gamma \left( a_j - \frac{i-1}{\alpha} + \tau_i \right) }{ \Gamma \left( a_j - \theta - \frac{i-1}{\alpha}  \right) }   \nonumber \\
&= \phi^{(\alpha)}(a_{1:p}) \times   (a_1)_{\tau}^{(\alpha)} (a_2)_{\tau}^{(\alpha)} \cdots (a_p)_{\tau}^{(\alpha)}  \nonumber 
\end{align}
and similarly
\begin{align}
(b_{1:q} - \theta )_{\tilde{\tau}}^{(\alpha)} &= \phi^{(\alpha)}(b_{1:q}) \times   (b_1)_{\tau}^{( \alpha )} (b_2)_{\tau}^{( \alpha)} \cdots (b_q)_{\tau}^{( \alpha )}  \nonumber \; .
\end{align}
Using these expressions in (\ref{eq:RHSLast}) and recalling (\ref{eq:pFqSeries}), we immediately have the left-hand side of (\ref{eq:pFqResult}), thereby completing the proof.

\end{proof}

\subsection{Proof of Corollary \ref{corr:Multi}}   \label{sec:Corr1Proof}

\begin{proof}
With $\beta = 2$ (i.e., $\alpha = 1$), the parameters in Lemma \ref{le:Main} simplify to:
\begin{align*} 
\omega^{(1)}  \left( \mathbfcal{X}, \mathbf{Y}, \mathbfcal{Z} \right) &= (-1)^{r(r-1)/2} \prod_{j=1}^r  \frac{ (n-j)! }{  (r-j)!  }   \prod_{i < j}^r ( z_j - z_i )^{2} \prod_{j=1}^r \left[  x_j^{-(n-r)} \prod_{s=1}^n (z_j - y_s )^{-1}  \right] 
\end{align*}
and
\begin{align*} 
\phi^{(1)} (a_{1:p})  = \prod_{j=1}^p  \prod_{i=1}^r  \prod_{k=1}^{n-r} \left( a_j - n + r - i + k  \right)  \; = \; \prod_{j=1}^p \prod_{i=0}^{r-1} \frac{(a_j-r+i)! }{ (a_j-n+i)! }   \; .
\end{align*}

We may also express the ${}_p F_q^{(1)}$ on the right-hand side of (\ref{eq:pFqResult}) in an equivalent determinant form by applying \cite[Lemma 5]{Chiani2}, leading to
\begin{align*}
{}_p F_q^{(1)} \left( a_{1:p} - n + r; b_{1:q}-n + r ; \mathbfcal{X}, \mathbfcal{Z}  \right) = K \cdot K_2 \, \frac{ {\rm det} \left( \bar{\mathbf{A}} \right) }{ \prod_{ i < j}^M (\tilde{x}_i - \tilde{x}_j)  \prod_{ i < j}^r (z_j - z_i )   }   
\end{align*}
where
\begin{align*} 
K = \prod_{i=1}^r (r - i)! \frac{  \prod_{j=1}^q  \prod_{k= 0}^{r-1}   (b_j - n +  k )! / (b_j - n )! }{  \prod_{j=1}^p  \prod_{k= 0}^{r-1}   (a_j - n +  k )! / (a_j - n )! } \; 
\end{align*}
and 
\begin{align*}
K_2 = (-1)^{r(r-1)/2} \prod_{\ell = 1}^M K_\ell (a_{1:p}, b_{1:q} )   \; .
\end{align*}
Here,  $\bar{\mathbf{A}} = \left[\bar{\mathbf{A}}_1^T,\dots,\bar{\mathbf{A}}_M^T\right]^T$ is an $r \times r$ matrix, 
with $k_\ell \times r$ submatrix $\bar{\mathbf{A}}_\ell$ taking entries
\begin{align*}
( \bar{\mathbf{A}}_\ell )_{i, j} = z_j^{k_\ell - i} {}_p F_q ( a_{1:p}-n+k_\ell-i+1, b_{1:q}-n+k_\ell-i+1; \tilde{x}_\ell z_j )  \; \; .
\end{align*}
Applying these simplifications in  (\ref{eq:pFqResult}), and noting that 
\begin{align*} 
\prod_{ i < j}^r (z_j - z_i )  = {\rm det} \left( z_j^{i-1} \right)_{i,j = 1}^r \,  ,
\end{align*}
we get
\begin{align*}
& {}_p F_q^{(1)} \left( a_{1:p}; b_{1:q}; \mathbf{X}, \mathbf{Y}   \right) = \frac{K( a_{1:p}, b_{1:q})}{\prod_{ i < j}^M (\tilde{x}_i - \tilde{x}_j )}  \prod_{\ell = 1}^M \frac{ K_{k_\ell} (a_{1:p}, b_{1:q} ) }{ \tilde{x}_\ell^{k_\ell (n-r) }} \\
& \hspace*{2cm} \times \frac{1}{r! } \oint_{{\rm C}} \cdots \oint_{{\rm C}} \frac{ {\rm det} \left( z_j^{i-1} \right)_{i,j = 1}^r  {\rm det} \left( \bar{\mathbf{A}} \right)  }{ \prod_{j=1}^r  \left[ 2 \pi \iota  \prod_{s=1}^n (z_j - y_s) \right] } {\rm d} z_1 \cdots {\rm d} z_r   \; .
\end{align*}
Finally, integrating using the Andreief identity \cite{Andreief}:
\begin{align} \label{eq:Andr}
\int \cdots \int {\rm det}\left(  f_i (x_j) \right)_{i,j=1}^r  {\rm det}\left(  g_i (x_j) \right)_{i,j=1}^r  \prod_{j=1}^r  {\rm d} \mu (x_j) =
r! \, {\rm det} \left( \int f_i (x) g_j (x) {\rm d}\mu(x)   \right)_{i,j=1}^r  
\end{align}
gives the result.
\end{proof}

\subsection{Proof of Theorem \ref{th:wishart} (Models A and B)} \label{sec:proofwishart}

We prove this result by virtue of the moment generating function (MGF) of the LSS, which is
\begin{align*}
\mathcal{M}(\lambda) & = \E\left [\e^{\lambda\sum_{k=1}^n f \left ( x_k/n \right)}\right] \text{.}
\end{align*}
Using \eqref{eq:density}, upon applying the transformations $x_j \rightarrow n x_j$ and $z_j \rightarrow n z_j$, we obtain
\begin{align}
\mathcal{M}(\lambda) & =  \frac{ K_{n}^{(A,B)}}{(2\imath\pi)^r}\oint_{\tilde{\rm C}}\cdots\oint_{\tilde{\rm C}} {\text{det}} \left ( z_j^{i-1}\right )_{i,j=1}^r\text{det}  \left ( l_i(nz_j)\right )_{i,j=1}^r Z_n(\lambda,z_1,\dots,z_r) \prod_{j=1}^r\dif z_j \label{eq:MGF-I}
\end{align}
where
\begin{align}
Z_n(\lambda,z_1,\dots,z_r)=\int_{\R_+^n} \prod_{i<j}^n(x_j-x_i)^2 \prod_{j=1}^n  \frac{x_j^{m-n}e^{-nx_j}}{\prod_{t=1}^r(z_t-x_j)} \e^{\lambda f(x_j)} \,\dif x_j 
 \label{eq:density2}
\end{align}
where the contour $\tilde{\rm C}$ encloses now counter-clockwise all scaled eigenvalues $x_1/n,\dots,x_n/n$ in its interior. This MGF expression has some structural similarity with that characterized in \cite{Passemier2}, which applied for the special case $r=1$, with the key differences being the second determinant (i.e., the one involving the $l_i (\cdot)$ functions),  and the product in the denominator of \eqref{eq:density2}. Nonetheless, here we may follow the same general approach, based on first employing a Coulomb fluid approximation of \eqref{eq:density2} followed by a saddlepoint approximation, with appropriate modifications.

\subsubsection{Coulomb fluid approximation of $Z_n(\lambda,z_1,\dots,z_r)$ in \eqref{eq:density2}}
It will be convenient to rewrite $Z_n(\lambda, z_1,\dots,z_r)$ in the equivalent form:
\begin{align}
Z_n(\lambda ,z_1,\dots,z_r) = \int_{\R_+^n} \e^{-\Phi(x_1,\dots,x_n)-\sum_{k=1}^n g(x_k)}\prod_{k=1}^n \dif x_k   \label{eq:Inew}
\end{align}
where
\begin{align}
g(x)= g(x, z_1,\dots,z_r) = - \lambda f(x) + \sum_{j=1}^r \ln(z_j-x) \text{,} \label{eq:g}
\end{align}
with
\begin{align*}
\Phi(x_1,\dots,x_n)=-2 \sum_{i<j}^n \ln|x_j-x_i| +n\sum_{j=1}^n v_0(x_j)
\end{align*}
where we have defined
\begin{align*}
v_0(x)=x-\left(\frac{m}{n}-1\right)\ln x  \text{.}
\end{align*}
Setting $g(x) = 0$ in (\ref{eq:Inew}), we also introduce
\begin{align*}
Z_n  = \int_{\R_+^n} \e^{-\Phi(x_1,\dots,x_n)}  \prod_{k=1}^n \dif x_k \text{,}
\end{align*}
which is simply a constant.

With this formulation, the results from \cite{Chen}, derived based on the Coulomb fluid method, now immediately suggest that as $n \to \infty$ with $m/n \to c$,  
\begin{align} 
Z_n(\lambda, z_1,\dots,z_r) \approx Z_n \e^{- \frac{S_1 (z_1,\dots,z_r) }{2} - S_2 (z_1,\dots,z_r)  }\label{eq:CF}
\end{align}
where
\begin{align}
S_1(z_1,\dots,z_r) &= \int_a^b g(x, z_1,\dots,z_r) \varrho(x, z_1,\dots,z_r) \, \dif x \label{eq:S1}\\
S_2(z_1,\dots,z_r)  &= n \int_a^b g(x, z_1,\dots,z_r) \tilde{\sigma}_0(x) \, \dif x \text{.} \label{eq:S2} 
\end{align}
Here $a= (1-\sqrt{c})^2$ and $b= (1+\sqrt{c})^2$, as defined in the theorem statement, whilst
\begin{align}
\tilde{\sigma}_0(x)=\frac{1}{2\pi}\frac{\sqrt{(b-x)(x-a)}}{x}\text{, } \quad x\in[a,b] \label{eq:sigma_0}
\end{align}
which is the Mar{\v c}enko-Pastur law (see \cite{Dyson2,Marcenko}).  Also, 

\begin{align}
\varrho(x, z_1,\dots,z_r )= -\lambda\rho_1 (x)+\rho_2  (x, z_1,\dots,z_r)   \label{eq:varrhoExp}
\end{align}
where 
\begin{align*}
\rho_1 (x)=\frac{1}{2\pi^2 \sqrt{(b-x)(x-a)}}\mathcal{P} \int_a^b \frac{\sqrt{(b-y)(y-a)}}{y-x} f'(y) \,\dif y
\end{align*}
and
\begin{align*}
\rho_2 (x, z_1,\dots,z_r) &= \sum_{j=1}^r\tilde{\rho}_2 (x, z_j)\text{,}
\end{align*}
where
\begin{align}
\tilde{\rho}_2 (x, z) &= \frac{1}{2\pi^2 \sqrt{(b-x)(x-a)}}\mathcal{P} \int_a^b \frac{ \sqrt{(b-y)(y-a)}}{y-x} \frac{1}{y-z} \,\dif y \text{, } \quad x\in[a,b]\nonumber\\
&= \frac{1}{2\pi\sqrt{(b-x)(x-a)}}\left(\frac{\sqrt{(z - a)(z - b)}}{z-x}-1\right)  \text{,}  \label{eq:rhoh}
\end{align}
where the last integration followed from \cite[Eq. (37)]{Passemier2}.

The objective is to obtain expressions for $S_1$ and $S_2$ in (\ref{eq:S1}) and (\ref{eq:S2}), respectively.  To this end, first consider $S_1$.  Substituting \eqref{eq:g} and \eqref{eq:varrhoExp} into \eqref{eq:S1} yields a quadratic equation in $\lambda$,
\begin{align}
S_1(z_1,\dots,z_r) =-\lambda^2 \sigma^2  - 2\lambda \sum_{j=1}^r\bar{\mu}(z_j) - \sum_{j=1}^r A_1(z_j)  \label{eq:S1f}
\end{align}
where $\sigma^2$ takes the form \eqref{eq:V1}, the linear coefficient $\bar{\mu}( \cdot )$ takes the form \eqref{eq:muVal} since (see \cite[Appendix B]{Passemier2} for details)
\begin{align*}
\bar{\mu}(z) &= \frac{1}{2}\int_a^b \left [  f(x) \tilde{\rho}_2 (x, z)   + \ln(z-x) \rho_1(x) \right] \dif x \\
&= \int_a^b   f(x) \tilde{\rho}_2 (x, z) \,\dif x \text{,} 
\end{align*}
whilst the constant term involves
 \begin{align*}
A_1(z) &=-\int_a^b   \ln(z-x) \tilde{\rho}_2  (x,  z  ) \, \dif x  \text{.} 
\end{align*}
Note that this last term is independent of the linear statistic $f( \cdot)$ and will not contribute to either the asymptotic mean or variance.

Focusing now on $S_2$, we substitute \eqref{eq:g} and \eqref{eq:sigma_0} into \eqref{eq:S2} to give
\begin{align}
S_2(z_1,\dots,z_r) &=-n \left ( \lambda \mu + \sum_{j=1}^r A_2( z_j ) \right )  \label{eq:S2f}
\end{align}
where $\mu$ takes the form \eqref{eq:mu}, whilst  
\begin{align*}
A_2( z ) = -\frac{1}{2\pi}\int_a^b \ln(z-x) \frac{\sqrt{(b-x)(x-a)}}{x}\, \dif x 
\end{align*}
is a constant which will contribute to the asymptotic mean in the sequel.

Combining \eqref{eq:CF} together with \eqref{eq:S1f} and \eqref{eq:S2f}, we obtain 
\begin{align}
Z_n(\lambda,z_1,\dots,z_r) \approx Z_n\e^{\lambda^2 \frac{\sigma^2}{2} + \lambda \left [ n\mu + \sum_{j=1}^r\bar{\mu}(z_j) \right ] + \frac{1}{2}\sum_{j=1}^rA_1(z_j) + n\sum_{j=1}^r A_2(z_j) } , \nonumber
\end{align}
which, upon substituting into \eqref{eq:MGF-I} gives, for large $n$,
 \begin{align}
\mathcal{M}(\lambda) \propto \mathcal{I}(\lambda)\e^{\lambda^2 \frac{\sigma^2}{2}+\lambda  n\mu}   \label{eq:MGFap2}
\end{align}
where
\begin{align*}
\mathcal{I}(\lambda) &=  \oint_{\tilde{\rm C}}\cdots\oint_{\tilde{\rm C}}  {\text{det}} \left ( z_j^{i-1}\right )_{i,j=1}^r{\text{det}} \left (l_i(nz_j)\right )_{i,j=1}^r  \e^{ \lambda \sum_{j=1}^r\bar{\mu}(z_j)  + \frac{1}{2}\sum_{j=1}^r A_1(z_j)  + n\sum_{j=1}^r A_2(z_j) }\prod_{j=1}^r\dif z_j  \\
&= \oint_{\tilde{\rm C}}\cdots\oint_{\tilde{\rm C}} {\text{det}} \left ( z_j^{i-1}\right )_{i,j=1}^r\text{det}\left (l_i(nz_j)\e^{ \lambda \bar{\mu}(z_j) + \frac{A_1(z_j)}{2}   + n A_2(z_j)}\right )_{i,j=1}^r  \prod_{j=1}^r\dif z_j\text{.} 
\end{align*}
Applying the Andreief identity \eqref{eq:Andr}, we obtain
\begin{align}
\mathcal{I}(\lambda) &= {\text{det}} \left ( \oint_{\tilde{\rm C}} z^{j-1}l_i(nz)\e^{ \lambda \bar{\mu}(z) + \frac{A_1(z)}{2}   + n A_2(z)}\, \dif z \right )_{i,j=1}^r\text{.} \label{eq:Ilambda}
\end{align}
With the MGF expressed in this form, we are now in a position to apply saddlepoint approximation techniques in order to deal with the determinant of contour integrals.  It is important to recognize that this determinant has dimension $r \times r$, which remains \emph{fixed} as $n$ is taken large.

\subsubsection{Saddlepoint approximation}

The saddlepoint approximation will be applied to each entry of the determinant above.  To this end, following \cite[Section 5.2.1--2]{Passemier2}, we have for large $n$ and $1\le i \le r$,
\begin{align} \label{eq:spapprox1}
\oint_{\tilde{\rm C}} z^{j-1}l_i(nz)\e^{ \lambda \bar{\mu}(z) + \frac{A_1(z)}{2}   + n A_2(z)}\, \dif z \approx  z_{0,i}^{j-1} \e^{\lambda\bar{\mu}(z_{0,i})+h(z_{0,i})}\text{,}
\end{align}
where
\begin{align*}
z_{0,i}=\begin{cases}
\frac{(1+c\delta_i)(1+\delta_i)}{\delta_i}, &  \text{for Model A} \\
\frac{(1+\nu_i)(c+\nu_i)}{\nu_i},&  \text{for Model B}
\end{cases}
\end{align*}
is the saddlepoint and $h(z_{0,i})$ is a function which does not depend on $\lambda$. 

\begin{remark}\label{rem:smodelAB} 
In order to find a saddlepoint outside $[a,b]$, we have to take the following specific branches for the square root $\sqrt{(z_{0,i}-a)(z_{0,i}-b)}$:
\begin{list}{$\bullet$}{\leftmargin=2em}
\item When $0<\delta_i\le 1/\sqrt{c}$ for Model A (resp. $0<\nu_i\le\sqrt{c}$ for Model B), the branch is chosen so that the signs of the real and imaginary part of $\sqrt{(z_{0,i}-a)(z_{0,i}-b)}$ match those of $z_{0,i}-c-1$; \\
\item When $\delta_i > 1/\sqrt{c}$ for Model A (resp. $\nu_i > \sqrt{c}$ for Model B), the signs are chosen to be opposite.
\end{list}
The square root in both cases then evaluates to the common form:
\begin{align*}
\sqrt{(z_{0,i}-a)(z_{0,i}-b)}=\begin{cases}
\frac{1-c\delta_i^2}{\delta_i}, &  \text{for Model A} \\
\frac{c}{\nu_i}-\nu_i&  \text{for Model B}
\end{cases} 
\end{align*}
This square root is also encountered in the theorem statement, for which the same branch should be taken as indicated above (see \cite{Passemier2} for more details).
\end{remark}

 Plugging \eqref{eq:spapprox1} into \eqref{eq:Ilambda}, we have
\begin{align*}
\mathcal{I}(\lambda) &\approx   {\text{det}} \left ( z_{0,j}^{i-1}\right )_{i,j=1}^r\e^{\lambda\sum_{j=1}^r\bar{\mu}(z_{0,j})+\sum_{j=1}^r h(z_{0,j})}
\end{align*}
as $n\rightarrow \infty$. Finally, using the above expression and omitting the leading determinant and all other terms which are independent of $\lambda$ (these are absorbed into a proportionality constant),  under the same large-$n$ asymptotics, we can rewrite \eqref{eq:MGFap2} as
\begin{align*}
\mathcal{M}(\lambda) \propto  \exp \left ( \lambda^2\frac{\sigma^2}{2}+\lambda\left (n\mu+ \sum_{j=1}^r\bar{\mu}(z_{0,j})\right ) \right ) \; .
\end{align*}
This is recognized as the MGF of a Gaussian distribution with mean $n\mu+ \sum_{j=1}^r\bar{\mu}(z_{0,j})$ and variance $\sigma^2$. \qed

\subsection{Proof of Theorem \ref{th:F} (Model C)} \label{sec:SaddleC}
The proof is similar to Theorem \ref{th:wishart}, thus here we give just a brief treatment. In this case, we will evaluate the MGF 
\begin{align*}
\mathcal{M}(\lambda) & = \E\left [\e^{\lambda\sum_{k=1}^n f \left ( x_k \right)}\right] 
\end{align*}
which upon using \eqref{eq:densityF} yields
\begin{align}
\mathcal{M}(\lambda) & =  \frac{ K_{n}^{(C)}}{(2\imath\pi)^r}\oint_{\rm C}\cdots\oint_{\rm C} {\text{det}} \left ( z_j^{i-1}\right )_{i,j=1}^r\text{det}  \left ( l_i(z_j)\right )_{i,j=1}^r Z_n(\lambda,z_1,\dots,z_r) \prod_{j=1}^r\dif z_j \label{eq:MGF-IF}
\end{align}
where
\begin{align*}
Z_n(\lambda, z_1,\dots,z_r) =   \int_{(0,1)^n}  \prod_{i<j}^n (\mathsf{f}_j-\mathsf{f}_i)^2\prod_{j=1}^n \frac{\mathsf{f}_j^{m_1-n}(1-\mathsf{f}_j)^{m_2-n}}{\prod_{t=1}^r (z_t-\mathsf{f}_j)}
{\rm e}^{\lambda f \left ( \frac{\mathsf{f}_j}{1-\mathsf{f}_j}\right)}\, \dif \mathsf{f}_j  
\end{align*}
and
\begin{align*}
l_i(z)&=\,_1F_1\left(m_1+m_2-n+1,m_1-n+1,n\nu_i z\right ) \text{.}
\end{align*}

Following the derivation of Theorem \ref{th:wishart}, in this case we obtain
\begin{align*}
v_0(x)&=(1-c_1)\ln x + (1-c_2)\ln (1-x)\\
\tilde{\sigma}_0(x)&=\frac{c_1+c_2}{2\pi} \frac{\sqrt{(b-x)(x-a)}}{x(1-x)}\text{, } \quad x\in[a,b] 
\end{align*}
with $a$ and $b$ defined as in the theorem statement. Using the results of \cite{Chen},  we have as $n\rightarrow \infty$ such that $m_1/n\rightarrow c_1$ and $m_2/n\rightarrow c_2$,
\begin{align} 
Z_n(\lambda, z_1,\dots,z_r) \approx Z_n \e^{- \frac{S_1 (z_1,\dots,z_r) }{2} - S_2 (z_1,\dots,z_r)  }\label{eq:CF1F1}
\end{align}
with
\begin{align*}
S_1 (z_1,\dots,z_r) =-\lambda^2 \sigma_\F^2  - 2\lambda\sum_{j=1}^r\bar{\mu}_\F(z_j) -\sum_{j=1}^r A_{\F,1}(z_j) 
\end{align*}
where $\sigma_\F^2$ takes the form \eqref{eq:V1F}, the linear coefficient $\bar{\mu}_\F( \cdot )$ takes the form \eqref{eq:muValF}, whilst the constant term is
 \begin{align*}
A_{\F,1}(z) &=-\int_a^b   \ln(z-x) \tilde{\rho}_2  (x,  z  ) \, \dif x\text{,} 
\end{align*}
where $\tilde{\rho}_2 (x, z)$ is given by \eqref{eq:rhoh}. This constant term $A_{\F,1}$ is independent of the LSS $f( \cdot)$ and will not contribute to either the asymptotic mean or variance.

For $S_2$, we have
\begin{align*}
S_2(z_1,\dots,z_r) &=-n\left (\lambda \mu_\F - \sum_{j=1}^r A_{\F,2}( z_j ) \right )
\end{align*}
where $\mu_\F$ takes the form \eqref{eq:muF}, whilst  
\begin{align*}
A_{\F,2}( z ) = -\frac{c_1+c_2}{2\pi}\int_a^b \ln(z-x)  \frac{\sqrt{(b-x)(x-a)}}{x(1-x)}\, \dif x  \text{.}
\end{align*}
Substituting \eqref{eq:CF1F1} into \eqref{eq:MGF-IF} we obtain that, as $n \rightarrow \infty$ with $m_1/n\rightarrow c_1$  and  $m_2/n\rightarrow c_2$,
\begin{align}
\mathcal{M}(\lambda) \propto   \mathcal{I}(\lambda) \e^{\lambda^2\frac{\sigma^2_\F}{2}+\lambda n\mu_\F}\label{eq:MGFap21F1}
\end{align}
with
\begin{align}
\mathcal{I}(\lambda) &=  \oint_{\rm C}\cdots\oint_{\rm C} {\text{det}} \left ( z_j^{i-1}\right )_{i,j=1}^r{\text{det}} \left (l_i(z_j)\right )_{i,j=1}^r  \e^{ \lambda \sum_{j=1}^r\bar{\mu}_\F(z_j)  + \frac{1}{2}\sum_{j=1}^r A_{\F,1}(z_j)  + n\sum_{j=1}^r A_{\F,2}(z_j) }\prod_{j=1}^r\dif z_j  \nonumber\\
&= \oint_{\rm C}\cdots\oint_{\rm C} \text{det} \left ( z_j^{i-1}\right )_{i,j=1}^r {\text{det}} \left (l_i(z_j)\e^{ \lambda \bar{\mu}_\F(z_j) + \frac{A_{\F,1}(z_j)}{2}   + n A_{\F,2}(z_j)}\right )_{i,j=1}^r   \prod_{j=1}^r\dif z_j\text{.} \nonumber
\end{align}
Applying the Andreief identity \eqref{eq:Andr}, we obtain
\begin{align}\label{eq:spapprox2}
\mathcal{I}(\lambda) &= {\text{det}} \left ( \oint_{\rm C} z^{j-1}l_i(z)\e^{ \lambda \bar{\mu}_\F(z) + \frac{A_{\F,1}(z)}{2}   + n A_{\F,2}(z)}\, \dif z \right )_{i,j=1}^r \text{.} 
\end{align}

As before, we apply the saddlepoint method to deal with the contour integrals inside the determinant. In particular, following \cite[Section 5.3]{Passemier2}, we have for large $n$ and $1\le i \le r$,
\begin{align}
\oint_{\rm C} z^{j-1}l_i(z)\e^{ \lambda\bar{\mu}_\F(z) + \frac{A_{\F,1}(z)}{2}   + n A_{\F,2}(z)}\, \dif z \approx z_{0,i}^{j-1} \e^{\lambda\bar{\mu}_\F(z_{0,i})+h(z_{0,i})}\text{,}\label{eq:AH2}
\end{align}
where
\begin{align*}
z_{0,i}=\frac{(1+\nu_i)(c_1+\nu_i)}{\nu_i(c_1+c_2+\nu_i)}
\end{align*}
is the saddlepoint and $h(z_{0,i})$ is a function which does not depend on $\lambda$.
\begin{remark}\label{rem:smodelC} 
In order to have a saddlepoint outside $[a,b]$, we have to take the following specific branches for the square root $\sqrt{(z_{0,i}-a)(z_{0,i}-b)}$:
\begin{list}{$\bullet$}{\leftmargin=2em}
\item When $0<\nu_i\le\tilde{c} := \frac{c_1+\sqrt{c_1 c_2(c_1+c_2-1)}}{c_2-1}$, the branch is chosen so that the signs of the real and imaginary parts match those of $1-c_1+(c_1+c_2) z_{0,i}$; \\
\item When $\nu_i > \tilde{c}$, the signs are chosen to be opposite.
\end{list}
The square root in both cases then evaluates to the common form:
\begin{align*}
\sqrt{(z_{0,i}-a)(z_{0,i}-b)}&=\frac{c_1(c_1+c_2)+2c_1\nu_i-(c_2-1)\nu_i^2}{\nu_i(c_1+c_2+\nu_i)(c_1+c_2)}\text{.}
\end{align*}
This square root is also encountered in theorem statement, for which the same branch should be taken as indicated above. 
\end{remark}

Applying \eqref{eq:spapprox2} in \eqref{eq:Ilambda} produces
\begin{align*}
\mathcal{I}(\lambda) &\approx  \text{det} \left ( z_{0,j}^{i-1}\right )_{i,j=1}^r\e^{\lambda\sum_{j=1}^r\bar{\mu}_\F(z_{0,j})+\sum_{j=1}^r h(z_{0,j})}
\end{align*}
as $n\rightarrow \infty$. Finally, using the above expression, omitting the leading determinant and all other terms which are independent of $\lambda$ (as before), we can rewrite \eqref{eq:MGFap21F1} for large $n$ as
\begin{align*}
\mathcal{M}(\lambda) \propto  \exp \left ( \lambda^2\frac{\sigma_\F^2}{2}+\lambda\left (n\mu_\F+ \sum_{j=1}^r\bar{\mu}_\F(z_{0,j})\right )  \right ) \; .
\end{align*}
This is recognized as the MGF of a Gaussian distribution with mean $n\mu_\F+ \sum_{j=1}^r\bar{\mu_\F}(z_{0,j})$ and variance $\sigma_\F^2$. \qed

\subsection{Derivations with spike multiplicities} \label{sec:ext}

Here, we will limit the discussion to Models A and B, with the extension to Model C being straightforward. 
Using the joint density in \eqref{eq:densityMult}, rather than \eqref{eq:density},  we may  follow the same steps as for the proof of Theorem \ref{th:wishart} in Section \ref{sec:proofwishart}, until the point just prior to where the saddlepoint approximation is applied. Indeed, we obtain the same formula as (\ref{eq:MGFap2}), i.e., 
 \begin{align}
\mathcal{M}(\lambda) \propto \mathcal{I}(\lambda)\e^{\lambda^2 \frac{\sigma^2}{2}+\lambda  n\mu}    \label{eq:MultiM}
\end{align}
with $\mu$ and $\sigma^2$ again given by  (\ref{eq:mu}) and (\ref{eq:V1}) respectively, but in this case $\mathcal{I}(\lambda)$ now admits
\begin{align}
\mathcal{I}(\lambda)&= \oint_{\tilde{\rm C}}\cdots\oint_{\tilde{\rm C}} {\text{det}} \left ( z_j^{i-1}\right )_{i,j=1}^r\text{det}(\tilde{\mathbf{A}})  \prod_{j=1}^r\dif z_j\text{,} \label{eq:Ilambdabef2}
\end{align}
where $\tilde{\mathbf{A}} = \left[\tilde{\mathbf{A}}_1^T,\dots,\tilde{\mathbf{A}}_M^T\right]^T$ is an $r \times r$ matrix, with $k_\ell \times r$ submatrix $\tilde{\mathbf{A}}_\ell$ taking entries
\begin{align*}
( \tilde{\mathbf{A}}_\ell )_{i, j} =
\begin{cases}
 z_j^{k_\ell - i}\e^{ \lambda \bar{\mu}(z_j) + \frac{A_1(z_j)}{2}   + n A_2(z_j)}\e^{  \frac{\delta_\ell}{\delta_\ell+1} n z_j } &\text{for Model A}\\
z_j^{k_\ell - i} \e^{ \lambda \bar{\mu}(z_j) + \frac{A_1(z_j)}{2}   + n A_2(z_j)}{}_0 F_1 ( m-n+k_\ell-i+1; n^2\nu_\ell z_j )&\text{for Model B,}
\end{cases}
\end{align*}
with the contour $\tilde{{\rm C}}$ defined as before. Now,  applying the Andreief-Heine identity \eqref{eq:Andr}, \eqref{eq:Ilambdabef2} becomes
\begin{align}
\mathcal{I}(\lambda) &= {\text{det}} (\hat{\mathbf{A}})\text{,} \label{eq:Imultdet}
\end{align}
where $\hat{\mathbf{A}} = \left[\hat{\mathbf{A}}_1^T,\dots,\hat{\mathbf{A}}_M^T\right]^T$ is an $r \times r$ matrix, with $k_\ell \times r$ submatrix $\hat{\mathbf{A}}_\ell$ taking entries
\begin{align}\label{eq:Ilambdamult}
( \hat{\mathbf{A}}_\ell )_{i, j} = 
\begin{cases}
 \oint_{\tilde{\rm C}} z^{k_\ell-i+j-1}\e^{ \lambda \bar{\mu}(z) + \frac{A_1(z)}{2}   + n A_2(z)}\e^{  \frac{\delta_\ell}{\delta_\ell+1} n z }\, \dif z  &\text{for Model A}\\
\oint_{\tilde{\rm C}} z^{k_\ell-i+j-1}\e^{ \lambda \bar{\mu}(z) + \frac{A_1(z)}{2}   + n A_2(z)}{}_0 F_1 ( m-n+k_\ell-i+1; n^2\nu_\ell z_j )\, \dif z &\text{for Model B.}
\end{cases}
\end{align}
The next step is to apply a saddlepoint approximation to the above integral.

{\bf A key issue:}  Naturally, we may try the same approach as used to prove Theorem \ref{th:wishart}, based on the \emph{same} saddlepoint approximation, which is effectively a \emph{leading-order} approximation in $n$.  We now show, however, that this approach is problematic. In particular, mirroring the same steps as before yields
\begin{align}
( \hat{\mathbf{A}}_\ell )_{i, j} \approx  z_{0,\ell}^{k_\ell-i+j-1} \e^{\lambda\bar{\mu}(z_{0,\ell})+h(z_{0,\ell})} =: (\check{\mathbf{A}}_\ell )_{i, j} \text{,}\label{eq:AHmult}
\end{align}
where
\begin{align}\label{eq:spoint}
z_{0,\ell}=\begin{cases}
\frac{(1+c\delta_\ell)(1+\delta_\ell)}{\delta_\ell} &  \text{for Model A} \\
\frac{(1+\nu_\ell)(c+\nu_\ell)}{\nu_\ell}&  \text{for Model B}
\end{cases} 
\end{align}
is the saddlepoint and $h(z_{0,\ell})$ is a function which does not depend on $\lambda$.  Now, consider an index $\ell$ for which $k_\ell>1$. In this case, the rows of the sub-matrix $(\check{\mathbf{A}}_\ell )_{i, j}$ are proportional to one another, and thus 
\begin{align}
\mathcal{I}(\lambda) \; = \; \text{det} (\check{\mathbf{A}} ) \; = \; 0.  \label{eq:ZeroDet}
\end{align} 
From the above argument, it is clear that, when there exist spike multiplicities, in order to capture the asymptotic behavior in $n$ of (\ref{eq:Ilambdabef2}) (and thus of the MGF),  a \emph{more refined saddlepoint approximation} is needed.

Indeed, more refined saddlepoint approximations are possible \cite{Olver,Bleistein}, which include correction terms in $n$; however, the complexity of the approximation also increases substantially with the addition of more terms (and it is also difficult to give a generic formula, which effectively would comprise an expansion in $n$ to arbitrary degree of accuracy).  Moreover, it turns out that the number of correction terms required to facilitate our analysis (i.e., to capture the non-zero leader order behavior of $\mathcal{I}(\lambda)$ in (\ref{eq:Ilambdabef2})) is tied to the number of coinciding spikes; i.e., as $k_\ell$ increases, then so does the required number of terms. For this reason, a general proof for arbitrary multiplicities of spikes is not forthcoming. Nonetheless, in order to demonstrate the procedure, in the following we will consider a saddlepoint approximation with first-order correction.  In this case, we demonstrate that the desired large-$n$ asymptotics are appropriately captured for $r=2$ and $r=3$, but \emph{not} for $r = 4$ (which requires subsequent refinement with additional correction terms).

\subsubsection{Saddlepoint approximation with first-order correction}

With the inclusion of a first-order correction term (see \cite{Olver}), for large $n$, the refined saddlepoint approximation applied to \eqref{eq:Ilambdamult} yields
\begin{align}
( \hat{\mathbf{A}}_\ell )_{i, j} &= \oint_{\tilde{\rm C}} \e^{-np_\ell(z)}q_{k_\ell-i+j-1}(z) \,\dif z\nonumber\\
& \approx \e^{-np_\ell(z_{0,\ell})}\sqrt{\frac{2\pi}{n}}\frac{q_{k_\ell-i+j-1}(z_{0,\ell})}{\sqrt{p_\ell''(z_{0,\ell})}}\left (1+ \Psi_{k_\ell-i+j-1}(z_{0,\ell}) \right )\text{,} \label{eq:AH}
\end{align}
where $z_{0,\ell}$ is the saddlepoint given in \eqref{eq:spoint}, whilst
\begin{align*}
\Psi_{k_\ell-i+j-1}(z) = \frac{1}{n} \frac{1}{2 p_\ell''(z)}\left ( \frac{q_{k_\ell-i+j-1}''(z)}{q_{k_\ell-i+j-1}(z)}-\frac{p_\ell'''(z)q_{k_\ell-i+j-1}'(z)}{p_\ell''(z)q_{k_\ell-i+j-1}(z)} + \frac{5p''''(z)^2}{12p_\ell''(z)}-\frac{p_\ell''''(z)}{4p_\ell''(z)}\right )
\end{align*}
for which
\begin{align*}
p_\ell(z)=
\begin{cases}
- ( \delta_\ell/(1+\delta_\ell)z + A_2(z))& \text{for Model A}\\
- ( \sqrt{t_\ell(z)} + (1-c)\ln \left ( c-1+ \sqrt{t_\ell(z)}\right ) + A_2(z)  & \text{for Model B,}
\end{cases}
\end{align*}
with $t_\ell(z)=(c-1)^2+4\nu_\ell z$. Also, $q_{k_\ell-i+j-1}=z^{k_\ell-i+j-1}q_\ell(z)$, where
\begin{align*}
q_\ell(z) &=
\begin{cases}
 \e^{ \lambda \bar{\mu}(z) + \frac{A_{1}(z)}{2}}& \text{for Model A}\\
 \e^{ \lambda \bar{\mu}(z) + \frac{A_{1}(z)}{2} + (c-1) \ln(2\sqrt{\nu_\ell}) + \ln((c-1)^2+4\nu_\ell z)}& \text{for Model B.}
\end{cases}
\end{align*}
One can verify that if $\Psi_{k_\ell-i+j-1}$ is omitted, then \eqref{eq:AH} reduces to (\ref{eq:AHmult}).

Plugging \eqref{eq:AH} into \eqref{eq:Imultdet}, we have, as $n\rightarrow \infty$,
\begin{align}
\mathcal{I}(\lambda) &\approx \sqrt{\frac{(2\pi)^r}{n^r}}  \,{\text{det}}  ( \acute{\mathbf{A}})\e^{\lambda\sum_{\ell=1}^M k_\ell\bar{\mu}(z_{0,\ell})+\sum_{\ell=1}^M k_\ell h_\ell(z_{0,\ell})}\text{,} \label{eq:Ifinal}
\end{align} 
where 
\begin{align*}
h_\ell(z)=
\begin{cases}
\frac{A_{1}(z)}{2}- np_\ell(z)-\frac{1}{2}\ln(p''_\ell(z))-\ln(2np''_\ell(z)) & \text{for Model A}\\
\frac{A_{1}(z)}{2} + (c-1) \ln(2\sqrt{\nu_l}) + \ln t(z)- np_\ell(z)-\frac{3}{2}\ln(p''_\ell(z))-\ln(2n)& \text{for Model B,}
\end{cases}
\end{align*}
which shows no dependence on  $\lambda$, whilst $\acute{\mathbf{A}} = [\acute{\mathbf{A}}_1^T,\dots,\acute{\mathbf{A}}_M^T]^T$ is an $r \times r$ matrix, with $k_\ell \times r$ submatrix $\acute{\mathbf{A}}_\ell$ taking entries
\begin{align*}
( \acute{\mathbf{A}}_\ell )_{i,j} = n \,  2 p_\ell''(z_{0,\ell}) z_{0,\ell}^{k_\ell - i +j-1} \left (1+ \Psi_{k_\ell-i+j-1}(z_{0,\ell}) \right ) \; .
\end{align*}

Using \eqref{eq:Ifinal}, omitting all terms which we know are independent of $\lambda$ (i.e., absorbing them into the proportionality constant), we can now rewrite \eqref{eq:MultiM} as
\begin{align*}
\mathcal{M}(\lambda) \propto  \exp \left ( \lambda^2\frac{\sigma^2}{2}+\lambda\left (n\mu+ \sum_{\ell=1}^M k_\ell\bar{\mu}(z_{0,\ell})\right ) - \ln \text{det}(\acute{\mathbf{A}} )   \right ) 
\end{align*}
for large $n$. 

From this, the $k$-th cumulant of $\sum_{k=1}^n f(x_k/n)$ is computed as
\begin{align*}
{\rm c}_k= \frac{\dif^k }{\dif \lambda^k}\left [\ln\mathcal{M}(\lambda)\right]_{|_{\lambda=0}}\text{.}
\end{align*}
The MGF above would represent a Gaussian distribution with mean $n\mu+ \sum_{\ell=1}^M k_\ell\bar{\mu}(z_{0,\ell})$ and variance $\sigma^2$, provided that 
\begin{align*}
\frac{\dif^k }{\dif \lambda^k}\left [ \ln( \text{det}(\acute{\mathbf{A}}) )\right]_{|_{\lambda=0}} = O \left (\frac{1}{n}\right )
\end{align*}
for $k\ge 1$.  We will prove this below for $r = 2$ and $r = 3$; whilst demonstrating the inability to capture the case $r= 4$.

\paragraph{\em Case of $r=2$} With spike multiplicity, we have $M=1$, whilst $k_1=2$; i.e., a single spike of multiplicity two. In this case, the determinant $\text{det} (\acute{\mathbf{A}})$ does not evaluate to zero, as obvserved for 
$\text{det} (\check{\mathbf{A}} )$ in (\ref{eq:ZeroDet}), but rather gives
\begin{align*}
\text{det} (\acute{\mathbf{A}}) &= (n  2 p_1''(z_{0,1}))^2 \, {\rm det} \left ( \begin{matrix}
 z_{0,1} \left (1+ \Psi_{1}(z_{0,1}) \right ) &  z_{0,1}^{2} \left (1+ \Psi_{2}(z_{0,1}) \right )\\
 1+ \Psi_{0}(z_{0,1}) &  z_{0,1} \left (1+ \Psi_{1}(z_{0,1}) \right )
\end{matrix}\right) \\
&=(n  2 p_1''(z_{0,1})z_{0,1})^2  \left [\left (1+ \Psi_{1}(z_{0,1}) \right )^2 - \left (1+ \Psi_{0}(z_{0,1}) \right ) \left (1+ \Psi_{2}(z_{0,1}) \right ) \right ]
\end{align*}
which, for both Models A and B, is of the form 
\begin{align*}
\text{det} (\acute{\mathbf{A}}) = A\lambda^2 + B\lambda + C_1+nC_2\text{,}
\end{align*}
where $A,B,C_1,C_2$ do not depend on $\lambda$ nor $n$. Consequently,
\begin{align*}
\frac{\dif^k }{\dif \lambda^k}\left [ \ln( \text{det}(\acute{\mathbf{A}}) )\right]_{|_{\lambda=0}} = \begin{cases}O \left (n^{-1} \right ) & 1\le k \le 2\\ 0 & k > 2
\end{cases}
\end{align*}
thereby establishing the desired Gaussianity.  Thus, for $r = 2$, Theorem \ref{th:wishart} still holds, regardless of whether the two spikes are distinct or coincident.

\paragraph{\em Case of $r=3$}  With spike multiplicites, we have two scenarios to consider: 

\begin{list}{$\bullet$}{\leftmargin=2em}
\item $M=1$, with $k_1=3$; i.e., a single spike of multiplicity three. In this case,
\begin{align*}
\text{det} (\acute{\mathbf{A}}) &= (n 2 p_1''(z_{0,1}))^3 {\rm det} \left( \begin{matrix}
 z_{0,1}^2 \left (1+ \Psi_{2}(z_{0,1}) \right ) &  z_{0,1}^{3} \left (1+ \Psi_{3}(z_{0,1}) \right ) & z_{0,1}^{4} \left (1+ \Psi_{4}(z_{0,1}) \right )\\
 z_{0,1} \left (1+ \Psi_{1}(z_{0,1}) \right ) &  z_{0,1}^{2} \left (1+ \Psi_{2}(z_{0,1}) \right ) & z_{0,1}^{3} \left (1+ \Psi_{3}(z_{0,1}) \right )\\
 1+ \Psi_{0}(z_{0,1}) &  z_{0,1} \left (1+ \Psi_{1}(z_{0,1})\right ) & z_{0,1}^{2} \left (1+ \Psi_{2}(z_{0,1}) \right )
\end{matrix}\right) \\
&= C\text{,}
\end{align*}
where $C$ does not depend on $\lambda$ nor $n$. The same form holds for Models A and B. Consequently,
\begin{align*}
\frac{\dif^k }{\dif \lambda^k}\left [ \ln( \text{det}(\acute{\mathbf{A}}) )\right]_{|_{\lambda=0}} = 0 \; \; \quad \text{ for all }k\ge 1 \; .
\end{align*}

\item $M=2$, with $k_1=2$ and $k_2=1$ (or alternatively, $k_1=1$ and $k_2=2$); i.e., one of the spikes has multiplicity two, a second (distinct) spike has multiplicity one. In this case,
\begin{align*}
\text{det} (\acute{\mathbf{A}}) &= 8n^3p_1''(z_{0,1})p_2''(z_{0,2})^2 \nonumber \\
& \hspace*{1cm} \times  {\rm det} \left( \begin{matrix}
 1+ \Psi_{0}(z_{0,1}) &  z_{0,1} \left (1+ \Psi_{1}(z_{0,1})\right ) & z_{0,1}^{2} \left (1+ \Psi_{2}(z_{0,1}) \right )\\
 z_{0,2} \left (1+ \Psi_{1}(z_{0,2}) \right ) &  z_{0,2}^{2} \left (1+ \Psi_{2}(z_{0,2}) \right ) & z_{0,1}^{3} \left (1+ \Psi_{3}(z_{0,2}) \right )\\
 1+ \Psi_{0}(z_{0,2}) &  z_{0,2} \left (1+ \Psi_{1}(z_{0,2})\right ) & z_{0,2}^{2} \left (1+ \Psi_{2}(z_{0,2}) \right )
\end{matrix}\right) \\
&= A\lambda^4 + B\lambda^3 + (C_1+nC_2)\lambda^2 +(D_1+nD_2)\lambda+E_1+nE_2+n^2E_3 \text{,}
\end{align*}
where $A,B,C_1,C_2,D_1,D_2,E_1,E_2,E_3$ do not depend on $\lambda$ nor $n$. The same form holds for Models A and B. Consequently,
\begin{align*}
\frac{\dif^k }{\dif \lambda^k}\left [ \ln( \text{det}(\acute{\mathbf{A}}) )\right]_{|_{\lambda=0}} = \begin{cases}O \left (n^{-1} \right ) & 1\le k \le 4\\ 0 & k > 4 
\end{cases}
\; .
\end{align*}
\end{list}
Thus, these results indicate that Theorem \ref{th:wishart} still holds for $r = 3$, regardless of whether all three of the spikes are distinct, or if some or all of them coincide.

\paragraph{\em Case of $r=4$} Here, we just consider the case $M=1$, $k_1=4$; i.e., a single spike of multiplicity four.  For this,
\begin{align*}
\text{det} (\acute{\mathbf{A}}) & =(2np_1''(z_{0,1}))^4\\
& \times  {\rm det} \left( \begin{matrix}
 z_{0,1}^{3} \left (1+ \Psi_{3}(z_{0,1}) \right ) & z_{0,1}^{4} \left (1+ \Psi_{4}(z_{0,1}) \right )& z_{0,1}^{5} \left (1+ \Psi_{5}(z_{0,1}) \right )& z_{0,1}^{6} \left (1+ \Psi_{6}(z_{0,1}) \right )\\
 z_{0,1}^2 \left (1+ \Psi_{2}(z_{0,1}) \right ) &  z_{0,1}^{3} \left (1+ \Psi_{3}(z_{0,1}) \right ) & z_{0,1}^{4} \left (1+ \Psi_{4}(z_{0,1}) \right )& z_{0,1}^{5} \left (1+ \Psi_{5}(z_{0,1}) \right )\\
 z_{0,1} \left (1+ \Psi_{1}(z_{0,1}) \right ) &  z_{0,1}^{2} \left (1+ \Psi_{2}(z_{0,1}) \right ) & z_{0,1}^{3} \left (1+ \Psi_{3}(z_{0,1}) \right )& z_{0,1}^{4} \left (1+ \Psi_{4}(z_{0,1}) \right )\\
 1+ \Psi_{0}(z_{0,1}) &  z_{0,1} \left (1+ \Psi_{1}(z_{0,1})\right ) & z_{0,1}^{2} \left (1+ \Psi_{2}(z_{0,1}) \right ) & z_{0,1}^{3} \left (1+ \Psi_{3}(z_{0,1}) \right )
\end{matrix}\right) \\
&= 0 \text{,}
\end{align*}
for both Models A and B. Thus, as observed previously for $\text{det} (\check{\mathbf{A}} )$ in (\ref{eq:ZeroDet}), yet again we encounter a zero-valued determinant.  

To handle the case $r=4$, further refinement of the saddlepoint expansion \eqref{eq:AH} is required. Taking the full expansion from \cite{Olver} applied to \eqref{eq:Ilambdamult} yields, as $n \rightarrow \infty$,
\begin{align}
( \hat{\mathbf{A}}_\ell )_{i, j} &= \oint_{\tilde{\rm C}} \e^{-np_\ell(z)}q_{k_\ell-i+j-1}(z) \,\dif z\nonumber\\
& \approx \e^{-np_\ell(z_{0,\ell})}\sqrt{\frac{2\pi}{np_\ell''(z_{0,\ell})}} \sum_{s=0}^{\infty}\frac{(2s)!}{2^s s!}\frac{R(s)}{(np_\ell''(z_{0,\ell}))^s} \label{eq:expgen}
\end{align}
where $z_{0,\ell}$ is the saddlepoint given in \eqref{eq:spoint}, whilst $R(s)$ is the residue at $z=z_{0,\ell}$ of
\begin{align*}
\frac{\sum_{k=0}^\infty \frac{(z-z_{0,\ell})^k}{k!}q_{k_\ell-i+j-1}^{(k)}(z_{0,\ell})}{(z-z_{0,\ell})^{2s+1}\left (1+\frac{2}{p_\ell''(z_{0,\ell})} \sum_{k=0}^\infty \frac{(z-z_{0,\ell})^k}{(2+k)!}p_{\ell}^{(2+k)}(z_{0,\ell})\right)^{s+\frac{1}{2}}}\text{,}
\end{align*}
and $f^{(k)}(\cdot)$ denotes the $k$-th derivative of the function $f(\cdot)$. By taking only the first two terms ($s=0$ and 1) in the expansion \eqref{eq:expgen}, we recover \eqref{eq:AH}. We calculate the term $s=3$ of the sum:
\begin{align*}
&R(3)=\frac{1}{8} \left [ \frac{2}{3}q^{(4)}-\frac{1}{9p_\ell''}\left (p_\ell^{(6)}q +p_\ell^{(5)}q' + 5\left (3p_\ell^{(4)}q''+4p_\ell^{(3)}q^{(3)} \right )\right )\right.\\
&+ \frac{7}{72(p_\ell'')^2} \left ( 5 (p_\ell^{(4)})^2q + 40p_\ell^{(3)}p_\ell^{(4)}q'+8 ( p_\ell^{(3)}(p_\ell^{(5)}q+5p_\ell^{(3)}q'')) \right )\\
& \left.- \frac{35}{36}\left( \frac{p_\ell^{(3)}}{p_\ell''}\right ) \left ( 3p_\ell^{(4)}q+4p_\ell^{(3)}q'\right )+\frac{385}{216} \left (\frac{p_\ell^{(3)}}{p_\ell''} \right )^4q\ \right]
\end{align*}
where $q^{(k)}:=q_{k_\ell-i+j-1}^{(k)}$ and the functions are evaluated at $z_{0,\ell}$.

We do not further pursue this here, as the derivation becomes more lengthy and unwieldy as the number of spikes $r$ (or, more precisely, the number of spike multiplicities) and the number of terms in the necessarily-refined saddlepoint expansion increase.  Nonetheless, whilst we do not provide a rigorous proof of this claim, our results suggest that one may take any desired number of spikes to coincide, and the results in Theorem \ref{th:wishart} still apply.  The same is true for Theorem \ref{th:F}.

\bibliographystyle{alpha}
\bibliography{biblio}

\newcommand{\etalchar}[1]{$^{#1}$}
\begin{thebibliography}{BDMN11}

\bibitem[And83]{Andreief}
M.~C. Andreief.
\newblock Note sur une relation entre les int\'egrales d\'efinies des produits
  des fonctions.
\newblock {\em M\'em. de la Soc. Sci. Bordeaux 2}, 1883.

\bibitem[And03]{Anderson}
T.~W. Anderson.
\newblock {\em An Introduction to Multivariate Statistical Analysis}.
\newblock Wiley Series in Probability and Statistics. Wiley-Interscience [John
  Wiley \& Sons], Hoboken, NJ, third edition, 2003.

\bibitem[AZ05]{AndersonCLT}
G.~W. Anderson and O.~Zeitouni.
\newblock A {CLT} for a band matrix model.
\newblock {\em Probab. Theory Relat. Fields}, 134(2):283--338, 2005.

\bibitem[BC05]{BasorChen}
E.~Basor and Y.~Chen.
\newblock Perturbed {H}ankel determinants.
\newblock {\em J. Phys. A.: Math. Gen.}, 38(47):10101--10106, 2005.

\bibitem[BDMN11]{Bianchi}
P.~Bianchi, M.~Debbah, M.~Maida, and J.~Najim.
\newblock Performance of statistical tests for single-source detection using
  random matrix theory.
\newblock {\em IEEE Trans. Inform. Theory}, 57(4):2400--2419, 2011.

\bibitem[BH86]{Bleistein}
N.~Bleistein and R.~A. Handelsman.
\newblock {\em Asymptotic Expansions of Integrals}.
\newblock Dover Publications Inc., New York, second edition, 1986.

\bibitem[BJYZ09]{Bai09}
Z.~D. Bai, D.~Jiang, J.-F. Yao, and S.~Zheng.
\newblock Corrections to {LRT} on large-dimensional covariance matrix by {RMT}.
\newblock {\em Ann. Statist.}, 37(6B):3822--3840, 2009.

\bibitem[BJYZ13]{Bai13}
Z.~D. Bai, D.~Jiang, J.-F. Yao, and S.~Zheng.
\newblock Testing linear hypotheses in high-dimensional regressions.
\newblock {\em Statistics}, 47(6):1207--1223, 2013.

\bibitem[BP11]{Bouchaud}
J.~Bouchaud and M.~Potters.
\newblock Financial applications of random matrix theory: a short review.
\newblock In G.~Akemann, J.~Baik, and P.~Di~Francesco, editors, {\em The
  {O}xford {H}andbook of {R}andom {M}atrix {T}heory}, pages 824--850. Oxford
  Univ. Press, Oxford, 2011.

\bibitem[BS04]{Bai04}
Z.~D. Bai and Jack~W. Silverstein.
\newblock C{LT} for linear spectral statistics of large-dimensional sample
  covariance matrices.
\newblock {\em Ann. Probab.}, 32(1A):553--605, 2004.

\bibitem[CD11]{CouilletBook}
R.~Couillet and M.~Debbah.
\newblock {\em Random Matrix Methods for Wireless Communications}.
\newblock Cambridge Univ. Press, New York, first edition, 2011.

\bibitem[CH13]{Couillet}
R.~Couillet and W.~Hachem.
\newblock Fluctuations of spiked random matrix models and failure diagnosis in
  sensor networks.
\newblock {\em IEEE Trans. Inform. Theory}, 59(1):509--525, 2013.

\bibitem[CI97]{Chen+Ismail}
Y.~Chen and M.~E. Ismail.
\newblock Thermodynamic relations of the {H}ermitian matrix ensembles.
\newblock {\em J. Phys. A}, 30(19):6633--6654, 1997.

\bibitem[CL98]{Chen}
Y.~Chen and N.~Lawrence.
\newblock On the linear statistics of {H}ermitian random matrices.
\newblock {\em J. Phys. A}, 31(4):1141--1152, 1998.

\bibitem[CM94a]{Chen+Manning}
Y.~Chen and S.~M. Manning.
\newblock Asymptotic level spacing of the {L}aguerre ensemble: A {C}oulomb
  fluid approach.
\newblock {\em J. Phys. A.: Math. Gen.}, 27(11):3615--3620, 1994.

\bibitem[CM94b]{chemancond}
Y.~Chen and S.~M. Manning.
\newblock Distribution of linear statistics in random matrix models (metallic
  conductance fluctuations).
\newblock {\em J. Phys.: Cond. Matter}, 6(16):3039--3044, 1994.

\bibitem[CWS10]{Chiani2}
M.~Chiani, M.~Z. Win, and H.~Shin.
\newblock {MIMO} networks: {T}he effects of interference.
\newblock {\em IEEE Trans. Inform. Theory}, 56(1):336--349, 2010.

\bibitem[DE01]{Diaconis}
P.~Diaconis and S.~N. Evans.
\newblock Linear functional of eigenvalues of random matrices.
\newblock {\em Trans. American Math. Society}, 353(7):2615--2633, 2001.

\bibitem[DEKV13]{Dubbs}
A.~Dubbs, A.~Edelman, P.~Koev, and P.~Venkataramana.
\newblock The beta-{W}ishart ensemble.
\newblock {\em J. Math. Phys.}, 54:083507, 2013.

\bibitem[DES07]{Dumitriu2007}
I.~Dumitriu, A.~Edelman, and G.~Shuman.
\newblock {MOPS:} {M}ultivariate orthogonal polynomials (symbolically).
\newblock {\em J. Symb. Computat.}, 42(6):587--620, 2007.

\bibitem[Dha13]{Dharmawansa}
P.~Dharmawansa.
\newblock Three problems related to the eigenvalues of complex non-central
  {W}ishart matrices with rank-1 mean.
\newblock {\em submitted to SIAM Journal on Matrix Analysis and Applications,
  ar{X}iv:1306.6566}, 2013.

\bibitem[DHL{\etalchar{+}}10]{Dumont}
J.~Dumont, W.~Hachem, S.~Lasaulce, P.~Loubaton, and J.~Najim.
\newblock On the capacity achieving covariance matrix for {R}ician {MIMO}
  channels: An asymptotic approach.
\newblock {\em IEEE Trans. Inform. Theory}, 56(3):1048--1069, 2010.

\bibitem[DJ14]{Dharmawansa2}
P.~Dharmawansa and I.~M. Johnstone.
\newblock Joint density of eigenvalues in spiked multivariate models.
\newblock {\em Preprint ar{X}iv:1406.0267}, 2014.

\bibitem[DM08]{Dean2008}
D.~S. Dean and S.~N. Majumdar.
\newblock Extreme value statistics of eigenvalues of {Gaussian} random
  matrices.
\newblock {\em Phys. Rev. E}, 77(4):041108, 2008.

\bibitem[DSP{\etalchar{+}}11]{Dahirel}
V.~Dahirel, K.~Shekhar, F.~Pereyra, T.~Miura, M.~Artyomov, S.~Talsania, T.~M.
  Allen, M.~Altfeld, M.~Carrington, D.~J. Irvine, B.~D. Walker, and A.~K.
  Chakraborty.
\newblock Coordinate linkage of {HIV} evolution reveals regions of
  immunological vulnerability.
\newblock {\em Proc. Natl. Acad. Sci.}, 108(28):11530--11535, 2011.

\bibitem[Dys62]{Dyson}
F.~J. Dyson.
\newblock Statistical theory of energy levels of complex systems {I-III}.
\newblock {\em J. of Math. Phys.}, 3(1):140--175, 1962.

\bibitem[Dys71]{Dyson2}
F.~J. Dyson.
\newblock An {I}sing ferromagnet with discontinuous long-range order.
\newblock {\em Comm. Math. Phys.}, 21(4):269--283, 1971.

\bibitem[For11]{Forrester}
P.~J. Forrester.
\newblock Probability densities and distributions for spiked {W}ishart
  $\beta$-ensembles.
\newblock {\em Preprint ar{X}iv:1101.2261}, 2011.

\bibitem[FUS10]{Fujikoshi}
Y.~Fujikoshi, V.~V. Ulyanov, and R.~Shimizu.
\newblock {\em Multivariate Statistics}.
\newblock Wiley Series in Probability and Statistics. John Wiley \& Sons, Inc.,
  Hoboken, NJ, 2010.
\newblock High-dimensional and large-sample approximations.

\bibitem[GN00]{Gupta}
A.~K. Gupta and D.~K. Nagar.
\newblock {\em Matrix Variate Distributions}.
\newblock Chapman \& Hall / {CRC}, Boca Raton, 2000.

\bibitem[GR89]{Gross89}
K.~I. Gross and D.~S.~P. Richards.
\newblock Total positivity, spherical series, and hypergeometric functions of a
  matrix argument.
\newblock {\em J. Approx. Theory}, 59(2):224--246, 1989.

\bibitem[Jam64]{James}
A.~T. James.
\newblock Distributions of matrix variates and latent roots derived from normal
  samples.
\newblock {\em Ann. Math. Statist.}, 35(2):475--501, 1964.

\bibitem[Joh01]{Johnstone}
I.~M. Johnstone.
\newblock On the distribution of the largest eigenvalue in principal components
  analysis.
\newblock {\em Ann. Statist.}, 29(2):295--327, 2001.

\bibitem[Kha70]{Khatri70}
C.~G. Khatri.
\newblock On the moments of traces of two matrices in three situations for
  complex multivariate normal populations.
\newblock {\em Sankhya, The Indian J. Statist., Ser. A}, 32:65--80, 1970.

\bibitem[KN08]{Kritchman}
S.~Kritchman and B.~Nadler.
\newblock Determining the number of components in a factor model from limited
  noisy data.
\newblock {\em Chem. Int. Lab. Syst.}, 94:19--32, 2008.

\bibitem[KN09]{Kritchman2}
S.~Kritchman and B.~Nadler.
\newblock Non-parametric detection of the number of signals: {H}ypothesis
  testing and random matrix theory.
\newblock {\em IEEE Trans. Signal Process.}, 57(10):3930--3941, 2009.

\bibitem[Koe06]{Koev}
A.~Koev, P.~Edelman.
\newblock The efficient evaluation of the hypergeometric function of a matrix
  argument.
\newblock {\em Math. Comput.}, 75(254):833--846, 2006.

\bibitem[LP09]{Lytova}
A.~Lytova and L.~Pastur.
\newblock Central limit theorem for linear eigenvalue statistics of random
  matrices with independent entries.
\newblock {\em Ann. Probab.}, 37(5):1778--1840, 2009.

\bibitem[Mac95]{MacDonald}
I.~G. Macdonald.
\newblock {\em Symmetric Functions and {H}all Polynomials}.
\newblock Oxford Univ. Press, New York, 1995.

\bibitem[Meh04]{Mehta}
M.~L. Mehta.
\newblock {\em Random Matrices}.
\newblock Pure and Applied Mathematics. Elsevier/Academic Press, Amsterdam,
  third edition, 2004.

\bibitem[Mo12]{Mo}
M.~Y. Mo.
\newblock Rank 1 real {W}ishart spiked model.
\newblock {\em Comm. Pure Appl. Math.}, 65(11):1528--1638, 2012.

\bibitem[MP67]{Marcenko}
V.~A. Mar{\v{c}}enko and L.~A. Pastur.
\newblock Distribution of eigenvalues in certain sets of random matrices.
\newblock {\em Mat. Sb. (N.S.)}, 72(114):507--536, 1967.

\bibitem[Mui82]{Muirhead}
R.~J. Muirhead.
\newblock {\em Aspects of Multivariate Statistical Theory}.
\newblock John Wiley \& Sons, New York, 1982.

\bibitem[NIFD02]{Naes}
T.~Naes, T.~Isaksson, T.~Fearn, and T.~Davies.
\newblock {\em User-Friendly Guide to Multivariate Calibration and
  Classification}.
\newblock NIR Publications, Chichester, 2002.

\bibitem[NS10]{Nadakuditi}
R.~R. Nadakuditi and J.~W. Silverstein.
\newblock Fundamental limit of sample generalized eigenvalue based detection of
  signals in noise using relatively few signal-bearing and noise-only samples.
\newblock {\em IEEE Journal of Sel. Topic in Signal Proc.}, 4(3):468--480,
  2010.

\bibitem[Olv97]{Olver}
Frank~W.J. Olver.
\newblock {\em Asymptotics and Special Functions}.
\newblock A K Peters Natick, MA., fourth edition, 1997.

\bibitem[OMH13]{Onatski}
A.~Onatski, M.~J. Moreira, and M.~Hallin.
\newblock Asymptotic power of sphericity tests for high-dimensional data.
\newblock {\em Ann. Statist.}, 41(3):1204--1231, 2013.

\bibitem[OMH14]{Onatski2}
A.~Onatski, M.~J. Moreira, and M.~Hallin.
\newblock Signal detection in high dimension: The multispiked case.
\newblock {\em Ann. Statist.}, 42(1):225--254, 2014.

\bibitem[Ona14]{Onatski3}
A.~Onatski.
\newblock Detection of weak signals in high-dimensional complex-valued data.
\newblock {\em Random Matrices: Theory Appl}, 3(1):1450001, 2014.

\bibitem[PMC14]{Passemier2}
D.~Passemier, M.~R. McKay, and Y.~Chen.
\newblock Asymptotic linear spectral statistics for spiked hermitian random
  matrix models.
\newblock {\em Submitted to Ann. Inst. Henri Poincar\'e Probab. Stat.,
  ar{X}iv:1402.6419}, 2014.

\bibitem[PY13]{Passemier}
D.~Passemier and J.-F. Yao.
\newblock Variance estimation and goodness-of-fit test in a high-dimensional
  strict factor model.
\newblock {\em Submitted to Statist. Sinica, ar{X}iv:1308.3890}, 2013.

\bibitem[PY14]{Passemier1}
D.~Passemier and J.~Yao.
\newblock Estimation of the number of spikes, possibly equal, in the
  high-dimensional case.
\newblock {\em J. Multivariate Anal.}, 127:173--183, 2014.

\bibitem[QLS{\etalchar{+}}14]{Quadeer2013}
A.~A. Quadeer, R.~H. Louie, K.~Shekhar, A.K. Chakraborty, {I-M.} Hsing, and
  M.~R. McKay.
\newblock Statistical linkage analysis of substitutions in patient-derived
  sequences of genotype 1a {Hepatitis C} virus non-structural protein 3 exposes
  targets for immunogen design.
\newblock {\em J. Virol.}, 88(13):7628--7644, 2014.

\bibitem[SM06]{Simon2006}
S.~H. Simon and A.~L. Moustakas.
\newblock Crossover from conserving to lossy transport in circular
  random-matrix ensembles.
\newblock {\em Phys. Rev. Lett.}, 96(13):136805, 2006.

\bibitem[Sta89]{Stanley}
R.~P. Stanley.
\newblock Some combinatorial properties of {J}ack symmetric functions.
\newblock {\em Adv. in Math.}, 77(1):76--115, 1989.

\bibitem[TAA11]{Torun}
M.~U. Torun, A.~N. Akansu, and M.~Avellaneda.
\newblock Portfolio risk in multiple frequencies.
\newblock {\em {IEEE} Signal Process. Mag.}, 5(28):61--71, 2011.

\bibitem[VMB08]{Vivo2008}
P.~Vivo, S.~N. Majumdar, and O.~Bohigas.
\newblock Distributions of conductance and shot noise and associated phase
  transitions.
\newblock {\em Phys. Rev. Lett.}, 101(21):216809, 2008.

\bibitem[Wan12]{Wang}
D.~Wang.
\newblock The largest eigenvalue of real symmetric, {H}ermitian and {H}ermitian
  self-dual random matrix models with rank one external source, part i.
\newblock {\em J. Stat. Phys.}, 146(4):719--761., 2012.

\bibitem[WSY13]{Wang3}
Q.~Wang, J.~W. Silverstein, and J.-F. Yao.
\newblock A note on the {CLT} of the {LSS} for sample covariance matrix from a
  spiked population model.
\newblock {\em Preprint ar{X}iv:1304.6164}, 2013.

\bibitem[WY13]{Wang2}
Q.~Wang and J.-F. Yao.
\newblock On the sphericity test with large-dimensional observations.
\newblock {\em Electron. J. Statist.}, 7:2164--2192, 2013.

\bibitem[Zhe12]{Zheng}
S.~Zheng.
\newblock Central limit theorems for linear spectral statistics of large
  dimensional {$F$}-matrices.
\newblock {\em Ann. Inst. Henri Poincar\'e Probab. Stat.}, 48(2):444--476,
  2012.

\end{thebibliography}
\end{document}